\documentclass[11pt]{amsart}
\usepackage{amssymb,amsmath,amsthm,mathrsfs,verbatim,ulem,cancel,color}
\usepackage[all, knot]{xy}
\usepackage{rotating}
\usepackage{cancel}
\usepackage{enumitem}
\newtheorem{theorem}{Theorem}
\newtheorem{lemma}[theorem]{Lemma}

\newtheorem{proposition}[theorem]{Proposition}

\theoremstyle{remark}

\newtheorem*{remark}{Remark}
\newtheorem*{remarks}{Remarks}

\theoremstyle{definition}

\newtheorem*{definition*}{Definition}
\numberwithin{theorem}{section}
\numberwithin{equation}{section}
\numberwithin{figure}{section}

\newcommand{\R}{\mathbb{R}}
\newcommand{\C}{\mathbb{C}}

\newcommand{\Q}{\mathbb{Q}}

\newcommand{\Z}{\mathbb{Z}}
\newcommand{\N}{\mathbb{N}}

\newcommand{\SL}{{\text {\rm SL}}}
\newcommand{\PSL}{{\text {\rm PSL}}}

\newcommand{\sgn}{\operatorname{sgn}}

\newcommand{\im}{\textnormal{Im}}

\def\H{\mathbb{H}}

\renewcommand{\SS}{\mathbb{S}}

\allowdisplaybreaks
\begin{document}

\title[Ramanujan-like formuls for Fourier coefficients]{Ramanujan-like formulas for Fourier coefficients of all meromorphic cusp forms }

\author{Kathrin Bringmann} 
\address{Mathematical Institute\\University of
Cologne\\ Weyertal 86-90 \\ 50931 Cologne \\Germany}
\email{kbringma@math.uni-koeln.de}
\author{Ben Kane}
\address{Department of Mathematics\\ University of Hong Kong\\ Pokfulam, Hong Kong}
\email{bkane@maths.hku.hk}
\date{\today}
\thanks{The research of the first author was supported by the Alfried Krupp Prize for Young University Teachers of the Krupp foundation and the research leading to these results has received funding from the European Research Council under the European Union's Seventh Framework Programme (FP/2007-2013) / ERC Grant agreement n. 335220 - AQSER. The research of the second author was supported by grant project numbers 27300314, 17302515, and 17316416 of the Research Grants Council.  Part of this research was conducted while the second author was visiting Max--Planck Institute of mathematics in Bonn and the University of Cologne, which are both greatly thanked for their support and hospitality.}

\begin{abstract}

In this paper, we investigate Fourier expansions of meromorphic modular forms.  Over the years, a number of special cases of meromorphic modular forms were shown to have Fourier expansions closely resembling the expansion of the reciprocal of the weight $6$ Eisenstein series which was computed by Hardy and Ramanujan.  By investigating meromorphic modular forms within a larger space of so-called polar harmonic Maass forms, we prove in this paper that all negative-weight meromorphic modular forms (and furthermore all quasi-meromorpic modular forms) have Fourier expansions of this type, granted that they are bounded towards $i\infty$.  
\end{abstract}

\maketitle

\section{Introduction and statement of results.}

In 1950, Petersson \cite{Pe1} used two-variable meromorphic Poincar\'e series $H_{\kappa}$ (defined in \eqref{eqn:Hdef} below) of real weight $\kappa\in\R_{>2}$ in one variable to classify meromorphic modular forms of dual weight $2-\kappa<0$ via the principal parts of their Laurent expansions in the other variable.  To give a rough idea of Petersson's construction, for forms with simple poles, he took linear combinations of the $H_{\kappa}$ functions and then determined when the linear combinations are modular in the second variable. From another perspective, one can use a pairing of Bruinier--Funke \cite{BruinierFunke} to view his classification as a classification of the so-called polar harmonic Maass forms (modular by construction) which happen to be meromorphic.  In this paper, for $k\in\N_{>1}$, we use this perspective to investigate coefficients of (negative-weight) $2-2k$ \begin{it}meromorphic cusp forms\end{it} on $\SL_2(\Z)$, those meromorphic modular forms $f$ for which $y^{1-k}|f(z)|$ decays towards all cusps. The study of Fourier coefficients of such forms was initiated by Hardy and Ramanujan \cite{HR3} in 1918, where they computed the expansion for the reciprocal of the weight $6$ Eisenstein series. Expansions for other meromorphic cusp forms which closely resemble Hardy and Ramanujan's formula were discovered and proven in \cite{BeBiYe,Bi,BKFC}.  The primary goal of this paper is to prove that the similarity in the shape of these Fourier expansions is a general phenomenon which holds for all meromorphic cusp forms.  Furthermore, all negative-weight \begin{it}quasi-meromorphic cusp forms\end{it}, products of powers of the weight 2 Eisenstein series $E_2$ times meromorphic cusp forms $f$, are also shown to have Fourier expansions of the same type. 

 Our investigation in this paper follows recent renewed interest in meromorphic cusp forms.  For example, mirroring a construction of certain cusp forms $f_{k,D}$ from binary quadratic forms of discriminant $D>0$ by Zagier \cite{ZagierRQ}, Bengoechea \cite{Bengoechea} defined meromorphic cusp forms $f_{k,D}$ from binary quadratic forms of discriminant $D<0$.  Duke and Jenkins \cite{DJ} studied traces of meromorphic modular forms and such functions are also of importance for constructing canonical lifts \cite{AGOR, Gu}.  The classification of meromorphic cusp forms in \cite{Pe1} also yields explicit formulas for the Fourier coefficients for them in the case of simple poles which are reminiscent of the formula first discovered by Hardy and Ramanujan \cite{HR3} for the reciprocal of the weight $6$ Eisenstein series.  In \cite{Pe1}, the Fourier coefficients appear as other Poincar\'e series evaluated at points in the upper half-plane, and similar results for a more general class of elliptic Eisenstein series were obtained by von Pippich \cite{vonPippich}.

Before describing the expansions of meromorphic cusp forms,  let us first recall what is known for \begin{it}weakly holomorphic modular forms\end{it}, those meromorphic modular forms whose poles (if any) are supported at cusps.  In work which gave birth to the Circle Method, Hardy and Ramanujan \cite{HR1, HR2} derived their famous asymptotic formula for the partition function $p(n)$, namely, as $n\to\infty$,
\begin{equation}\label{eqn:partitiongrowth}
p(n)\sim \frac{1}{4n\sqrt{3}} e^{\pi \sqrt{  \frac{2n}3}  }.
\end{equation}
Rademacher \cite{Rad} then perfected the method to derive the exact formula
\begin{equation}\label{Radformula}
p(n)= 2 \pi (24n-1)^{-\frac{3}{4}} \sum_{k =1}^{\infty}\frac{A_k(n)}{k} I_{\frac{3}{2}}\left( \frac{\pi\sqrt{24n-1}}{6k}\right).
\end{equation}
Here $I_{l}(x)$ is the $I$-Bessel function of order $l$ and $A_k(n)$ denotes the Kloosterman sum
\begin{displaymath}
A_k(n):=\frac{1}{4} \sqrt{\frac{k}{3}} \sum_{\substack{h \pmod {24k}\\
h^2 \equiv -24n+1 \pmod{24k}}}  \chi_{12}(h) 
e\left(\frac{h}{12k}\right),
\end{displaymath}
where $e(\alpha):=e^{2\pi i \alpha}$ and $\chi_{12}(h):=\left(\frac{12}{h}\right)$. A key ingredient of the proof of \eqref{Radformula} is the fact that the partition generating function is 
essentially the reciprocal of a modular form with no roots in the upper half-plane, but which vanishes at the cusp $i\infty$ instead.  To be more precise, the function 
\begin{displaymath}
P(z):=\sum_{n=0}^{\infty}p(n)e^{2\pi i\left(n-\frac{1}{24}\right)z}=
e^{-\frac{\pi i z}{12}}\prod_{n=1}^{\infty}\frac{1}{1-e^{2\pi i nz}},
\end{displaymath}
is a weight $-1/2$ weakly holomorphic modular form.  Rademacher and Zuckerman \cite{RZ, Zu1, Zu2} subsequently generalized \eqref{Radformula} to obtain exact formulas for the coefficients of all weakly holomorphic modular forms of negative weight.  Using  modern techniques, a new formula for $p(n)$ as a (finite) trace of a certain weak Maass form evaluated at the CM points of discriminant $1-24n$ modulo the action of $\Gamma_0(6)$ was recently proven by Bruinier and Ono \cite{BruinierOno}.  Much in the same way that the sum \eqref{Radformula} restricted to $k\ll \sqrt{n}$ gives a very good asymptotic approximation to $p(n)$, Masri \cite{Masri} used Bruinier and Ono's result to obtain
 a good asymptotic approximation to $p(n)$ with a shorter sum.

Much less is known about Fourier coefficients of meromorphic cusp forms.  Hardy and Ramanujan \cite{HR3} considered the special case that the form has a unique simple pole modulo the action of $\SL_2 (\Z)$. In particular, they found a formula for the reciprocal of the weight $6$ Eisenstein series $E_6$.  Ramanujan \cite{RaLost} then stated further formulas for other meromorphic functions, but, as usual for his writing, did not provide a proof. His claims concerning meromorphic cusp forms with simple poles were then subsequently proven by Bialek in his Ph.D. thesis written under Berndt \cite{Bi}.  However, the problem of determining the Fourier coefficients of meromorphic forms becomes more difficult as the order of the poles increases. Berndt, Bialek, and Yee \cite{BeBiYe} were first to explicitly compute the Fourier coefficients of meromorphic cusp forms with second-order poles, resolving the last of Ramanujan's claims about their coefficients.  In \cite{BKFC}, the authors used various different Poincar\'e series to treat coefficients of a wide variety of meromorphic cusp forms, including forms with third-order poles.  This paper addresses the question for meromorphic cusp forms with arbitrary order poles, showing that these expansions all have the same shape, which we next describe. 

Formulas for Fourier coefficients of meromorphic modular forms look very different from \eqref{Radformula}.  To expound upon one example, we state the result for the reciprocal of the weight $4$ Eisenstein series $E_4$. In this case, we have for $y>\sqrt{3}/2$ 
\begin{equation*}
 \frac{1}{E_4(z)} = \sum_{n=0}^{\infty} \beta_n e^{2\pi i  n z}\quad \text{with}
\end{equation*}
\begin{equation}\label{bn2}
 \beta_n := \frac{6}{E_6(\rho)}\sum_{(\lambda)}\sum_{(c,d)} \frac{h_{(c,d)}(n)}{\lambda^3} e^{\frac{\pi n \sqrt{3}}{\lambda}}.
\end{equation}
Here $\rho:=e^{\frac{\pi i}{3}}$, $(c,d)$ runs over distinct solutions to $\lambda=c^2 -cd +d^2$, where $\lambda = 3^a \prod_{j=1}^{r}p_j^{a_j}$ with $a\in\{0,1\}, p_j$ denoting primes of the form $6m+1$, and $a_j \in \N_0$.  We do not clarify what distinct means in this paper, but the definition turns out to be equivalent to factoring out by the action of units in $\mathcal{O}_{\Q(\rho)}$.
 Finally, we let $h_{(1,0)}(n):=(-1)^{n}/2$, $h_{(2,1)}(n):=1/2$, and for $\lambda \geq 7$ we let $(a,b)\in\Z^2$ be any solution to $ad-bc=1$ and set 
\begin{equation*}
 h_{(c,d)}(n):=   \cos\left( (ad+bc -2ac-2bd )\frac{\pi n}{\lambda}-6 \arctan\left(\frac{c\sqrt{3}}{2d-c}\right)\right).
\end{equation*}

One sees directly that the main asymptotic term of $\beta_n$ comes from $\lambda=1$ in \eqref{bn2}.  One notes that the growth is significantly faster than in \eqref{eqn:partitiongrowth}.  This is explained by the fact that the Fourier expansions of weakly holomorphic modular forms converge for all $z\in\H$, whereas the Fourier expansion for a meromorphic modular form only converges for those $z$ with sufficiently large imaginary part, and hence the coefficients must grow exponentially.
 
Comparing \eqref{bn2} with the Fourier expansions of other meromorphic cusps forms with poles at elliptic fixed points $\mathfrak{z}$, it is striking that all known examples are linear combinations of the series 
\begin{equation}\label{Fklr}
F_{k,j,r}(\mathfrak{z},z):= \mathfrak{z}_2^{-j}\sum_{m=0}^{\infty} \;\sideset{}{^*}\sum_{\mathfrak{b}\subseteq\mathcal{O}_{\Q(\mathfrak{z})}} \frac{C_{k}\left(\mathfrak{b},m\right)}{N(\mathfrak{b})^{\frac{k}{2}-j}} (4\pi m)^r e^{\frac{2\pi  mv}{N(\mathfrak{b})}} e^{2\pi i m z},
\end{equation}
where $C_{k}$ is given by \eqref{eqn:C6} and \eqref{eqn:C4}, $N$ is the norm in $\mathcal{O}_{\Q (\mathfrak{z})}$, the sum on $\mathfrak{b}$ runs over primitive ideals, and $\mathfrak{z}=\mathfrak{z}_1+i\mathfrak{z}_2$ throughout.  

The appearance of the function $F_{k,j,r}$ in these examples is no accident; this paper utilizes the fact that the functions $H_{2k}$ from Petersson's classification are ``meromorphic parts'' of certain polar harmonic Maass forms $y^{2k-1}\Psi_{2k}(\mathfrak{z},z)$ (defined in \eqref{eqn:Psidef} below) to show that such an expansion is quite general.  In order to do so, we construct a basis of polar harmonic Maass forms and show that their Fourier coefficients are all of the shape \eqref{Fklr}.  To give an explicit version of this formula, we require a basis of weight $2-2k$ polar harmonic Maass forms with poles at $z=\tau_0$ whose meromorphic parts are given by $R_{2k,\mathfrak{z}}^{n}\left[H_{2k}(\mathfrak{z},z)\right]_{\mathfrak{z}=\tau_0}$ (see Proposition \ref{prop:basis}), where $R_{\kappa,\mathfrak{z}}:=2i \frac{\partial}{\partial \mathfrak{z}}+\frac{\kappa}{\mathfrak{z}_2}$ is the \begin{it}Maass raising operator\end{it}, and iterated raising is defined by $R_{\kappa,\mathfrak{z}}^n := R_{\kappa+2n-2,\mathfrak{z}}\circ\cdots \circ R_{\kappa,\mathfrak{z}}$.  Throughout the paper, we also write $\tau_{l}=u_{l}+iv_{l}$, denote by $\omega_{\tau_l}$ the order of the subgroup of $\PSL_2(\Z)$ fixing $\tau_l$, and write $\mathfrak{z}$ and $z$ dependencies only if needed.  Repeated raising has appeared in a number of contexts.  For example, Bol \cite{Bol} showed that $R_{2-2k}^{2k-1}$ is the same as holomorphic differentiation $2k-1$ times, Bruinier--Ono--Rhoades \cite{BruinierOnoRhoades} studied the image of $R_{2-2k}^{2k-1}$ on harmonic Maass forms, and Duke--Jenkins \cite{DJ} raised to weight 0 while constructing a Zagier lift between integral and half-integral weight weakly holomorphic modular forms. In this paper, the raising operator is applied in the $\mathfrak{z}$ variable to increase the order of poles in the $z$ variable.
\begin{theorem}\label{Fouriercoefficients}
Let $f$ be a weight $2-2k<0$ meromorphic cusp form with poles only at $\tau_0=i$ or $\tau_0=\rho$ modulo $\SL_2(\Z)$.  Then, for $y>v_0$, $f(z)$ has a Fourier expansion which is a linear combination of the series $F_{k_l,j,r}(\tau_0,z)$ with $k_l\in \N$, $j\in \N_0$, and $r\in \N_0$.  

In particular, if 
$$
f(z) = \sum_{n=0}^{n_0} a_{n} R_{2k,\mathfrak{z}}^{n}\left[H_{2k}\left(\mathfrak{z},z\right)\right]_{\mathfrak{z}=\tau_{0}},
$$
then, for $y>v_0$, we have the following Fourier expansion
$$
f(z) = 2\omega_{\tau_0}\sum_{n=0}^{n_0} a_{n} \sum_{j=0}^{n} \frac{(2k+n-1)!}{(2k+n-1-j)!}\binom{n}{j} F_{2k+2n,j,n-j}(\tau_0,z).
$$
\end{theorem}
\begin{remarks}
\noindent

\noindent
\begin{enumerate}[leftmargin=*]
\item
Related functions occurred when studying coefficients of certain meromorphic Jacobi forms \cite{BF} arising from characters introduced by Kac and Wakimoto \cite{KW}.
\item
In the introduction we restrict to meromorphic cusp forms whose only poles occur at $i$ or $\rho$ for cosmetic reasons, but our results are more general. Namely, Theorem \ref{thm:FourierCoeff} implies that all meromorphic cusp forms have Fourier expansions of similar shapes.  These Fourier expansions may be rewritten with an algebraic structure if the poles are specifically at $i$ or $\rho$. 
Moreover, Theorem \ref{thm:FourierCoeff} is proven by viewing meromorphic cusp forms within a larger space of polar harmonic Maass forms.  The meromorphic part of these polar harmonic Maass forms also have Fourier expansions of the same shape.
\item
Theorem \ref{Fouriercoefficients} further states that every meromorphic cusp form of weight $2-2k$ (with $k$ fixed here) is actually a linear combination of the more restrictive sums ($j\leq n$)
$$
v_0^{-j}\sum_{m=0}^{\infty} \;\sideset{}{^*}\sum_{\mathfrak{b}\subseteq\mathcal{O}_{\Q(\tau_0)}} \frac{C_{2k+2n}\left(\mathfrak{b},m\right)}{N(\mathfrak{b})^{k}} \left(\frac{4\pi m}{N(\mathfrak{b})}\right)^{n-j} e^{\frac{2\pi  mv_0}{N(\mathfrak{b})}} e^{2\pi i m z}.
$$
\item 
In addition to the Fourier coefficients of meromorphic modular forms, it is natural to consider elliptic expansions around points in the upper half-plane (see \eqref{eqn:ellexp}).  In particular, J. Zhou has pointed out that elliptic expansions of certain meromorphic quasi-modular forms have recently appeared in the Gromov-Witten theory of Calabi-Yau varieties \cite{ABK,ASYZ}, in which the coefficients contain rich enumerative information.
\end{enumerate}
\end{remarks}
In addition to formulas for meromorphic modular forms, Ramanujan conjectured \cite{RaLost} an expression for $E_2/E_6$, $E_2^2/E_6$, and $E_2/E_4$ as linear combinations of the functions $F_{k,j,r}$.  
These are once again part of a general framework, this time for quasi-meromorphic cusp forms.
\begin{theorem}\label{QuasiTheorem}
If $f$ is a weight $2-2k$ meromorphic cusp form, $n\in\N_0$ satisfies $2-2k+2n<0$, and all of the poles of $f$ are at $\mathfrak{z}=i$ or $\mathfrak{z}=\rho$, then $E_2^n f$ is a linear combination of $F_{2k_l\omega_{\mathfrak{z}},j,r}$ with $k_l \in \N$, $j\in \N_0$, and $r\in \N_0$.
\end{theorem}
\begin{remarks}
\noindent
\begin{enumerate}[leftmargin=*]
\item
A comment about the role of the functions $E_2^nf$ in the framework of modular 
forms is in order.  Although $E_2$ is not itself modular, it has a weight $2$ modular completion $\widehat{E}_2(z):=E_2(z)-\frac{3}{\pi y}$.  Hence, if $f$ is a weight $2-2k$ meromorphic modular form, then $\widehat{E}_2^n f$ is a weight $2n+2-2k$ \begin{it}almost meromorphic modular form\end{it}, i.e., a function which satisfies weight $2n+2-2k$ modularity and has an expansion of the type
\begin{equation}\label{eqn:almostmeroexp}
\sum_{j=0}^r y^{-j} f_j(z),
\end{equation}
where the functions $f_j$ are meromorphic.  However, the \begin{it}meromorphic part\end{it} $f_0$ of $\widehat{E}_2^nf$ is precisely $E_2^n f$.  Furthermore, the space of almost meromorphic modular forms of weight $2n+2-2k$ is spanned by $\widehat{E}_2^{n-m}g$, where $g$ is a weight $2m+2-2k$ meromorphic modular form.  Investigating almost meromorphic modular forms leads to Theorem \ref{QuasiTheorem}.
\item
Again, the restriction that all of the poles of $f$ are at $i$ or $\rho$ is done for cosmetic purposes.  In Theorem \ref{thm:EpowFCgen}, we show that all quasi-meromorphic cusp forms have similar Fourier expansions.
\end{enumerate}
\end{remarks}

Note that Theorem \ref{Fouriercoefficients} gives an explicit formula for the coefficients of a meromorphic cusp form written in terms of the meromorphic parts $R_{2k,\mathfrak{z}}^n \left[H_{2k}(\mathfrak{z},z)\right]_{\mathfrak{z}=\tau_0}$ of the basis elements.  In general, our results are explicit in the sense that one essentially obtains the Fourier coefficients of an arbitrary form once it is written as a linear combination of certain basis elements.  Since one may compute such basis representation given only the principal parts at (finitely many) points in $\H$, this closely mirrors the situation for weakly holomorphic modular forms, where one obtains the Fourier coefficients once the principal part of its Fourier expansion has been determined. Coefficients of new explicit examples are obtained by carrying out the calculations necessary to compute their representations with the given basis.  We work out the explicit examples $1/E_6^4$ and $E_2/E_6^4$ in Theorems \ref{thm:1/E6^4} and \ref{thm:E2/E6^4}, respectively, to give the flavor of such calculations.  If the meromorphic cusp form has simple poles, then a uniform result is obtained for all cases simultaneously.  In addition to the cases conjectured by Ramanujan, this yields the new following formula for the coefficients of $E_2^n/E_{10}$ for $0\leq n\leq 4$.  
\begin{theorem}\label{thm:E2^n/E10}
For $0\leq n\leq 4$ and $m\in\N_0$, the $m$th Fourier coefficient of $E_2^n/E_{10}$ equals
$$
\frac{4\left(\frac{3}{\pi}\right)^{n}}{E_4^3(i)} \;\sideset{}{^*}\sum_{\mathfrak{b}\subseteq\mathcal{O}_{\Q(i)}} \frac{C_{12}(\mathfrak{b},m)}{N(\mathfrak{b})^{6-n}} e^{\frac{2\pi m}{N(\mathfrak{b})}}+ \frac{6\left(\frac{2\sqrt{3}}{\pi}\right)^{n}}{E_6^2(\rho)} \;\sideset{}{^*}\sum_{\mathfrak{b}\subseteq\mathcal{O}_{\Q(\rho)}} \frac{C_{12}(\mathfrak{b},m)}{N(\mathfrak{b})^{6-n}} e^{\frac{\sqrt{3}\pi m}{N(\mathfrak{b})}}.
$$
\end{theorem}
\begin{remark}
By the corollary following Proposition 27 of \cite{123}, one can explicitly rewrite $E_4(i)=\frac{12}{(8\pi)^2}(\frac{\Gamma(\frac{1}{4})}{\Gamma(\frac{3}{4})} )^{4}$ and $E_6(\rho)=\frac{24\sqrt{3}}{(6\pi)^3}(\frac{\Gamma(\frac{1}{3})}{\Gamma(\frac{2}{3})})^9$. 
\end{remark}
The paper is organized as follows.  In Section \ref{sec:prelim}, we introduce weak Maass forms and polar Maass forms and construct a basis of harmonic Maass forms via Poincar\'e series.  In Section \ref{sec:conj}, we show that $H_{2k}$ is the meromorphic part of the polar harmonic Maass forms $y^{2k-1}\Psi_{2k}(\mathfrak{z},z)$.   Section \ref{sec:Fouriercoeff} is devoted to the proof of Theorem \ref{Fouriercoefficients}, and we prove Theorem \ref{QuasiTheorem} in Section \ref{sec:Quasi}.  The Fourier coefficients of the explicit examples $1/E_6^4$ and $E_2/E_6^4$ are computed in Section \ref{sec:examples} and we conclude the paper with Theorem \ref{thm:E2^n/E10}.

\section{Preliminaries}\label{sec:prelim}

\subsection{Polar harmonic Maass forms}\label{2.1}
In this section, we recall the definition of harmonic Maass forms and also introduce harmonic Maass forms which may have poles in $\H$.  

We begin by defining harmonic Maass forms, which Bruinier and Funke introduced in \cite{BruinierFunke}.
\begin{definition*}
For $\kappa \in\R$ with $\kappa < 1$, a congruence subgroup $\Gamma\subseteq \SL_2(\Z)$, and a multiplier system $\nu$ with $|\nu|=1$, a \begin{it}harmonic Maass form\end{it} of weight $\kappa$ on $\Gamma$ with multiplier $\nu$ is a real analytic function $\mathcal{F}:\H\to\C$ satisfying the following conditions:
\noindent

\noindent
\begin{enumerate}[leftmargin=*]
\item 
For every $M=\left(\begin{smallmatrix}a&b\\ c&d\end{smallmatrix}\right)\in\Gamma$, we have $\mathcal{F}|_{\kappa,\nu}M=\mathcal{F}$, where 
\begin{equation*}
\mathcal{F}(z)|_{\kappa,\nu}M  :=\left(cz+d\right)^{-\kappa} \nu(M)^{-1}
\mathcal{F}
(Mz).
\end{equation*}
For $(cz+d)^{-\kappa}$, we throughout use the principal branch of the logarithm.
\item 

The function $\mathcal{F}$ is annihilated by the \begin{it}weight $\kappa$ hyperbolic Laplacian\end{it} 
$$
\Delta_{\kappa}:=-y^2\left(\frac{\partial^2}{\partial x^2}+\frac{\partial^2}{\partial y^2}\right)+i\kappa y\left(\frac{\partial}{\partial x}+i\frac{\partial}{\partial y}\right).
$$
\item 
The function $\mathcal{F}$ grows at most linear exponentially towards cusps of $\Gamma$.
\end{enumerate}
If one allows in condition (2) a general eigenvalue under $\Delta_{\kappa}$, then one obtains a \begin{it}weak Maass form.\end{it}
\end{definition*}
We denote the space of harmonic Maass forms by $H_{\kappa}(\Gamma,\nu)$ and the subspace of weakly holomorphic modular forms by $M_{\kappa}^!(\Gamma,\nu)$, omitting the subgroup and multiplier whenever  $\Gamma=\SL_2(\Z)$ and $\nu$ is trivial.  Harmonic Maass forms have singularities at the cusps; if one additionally allows singularities in $\H$, then one obtains polar harmonic Maass forms.  
\begin{definition*}
Assume that $\kappa,\Gamma$, and $\nu$ satisfy the conditions 
given in the definition of harmonic Maass forms.  
 A \begin{it}polar harmonic Maass form\end{it} of weight $\kappa$ on $\Gamma$ with multiplier $\nu$ is a function $\mathcal{F}:\H\to\C$ which is real analytic outside of a discrete set of points satisfying the following: 
\noindent

\noindent
\begin{enumerate}[leftmargin=*]
\item 
For every $M\in\Gamma$, we have $\mathcal{F}|_{\kappa,\nu}M=\mathcal{F}$.
\item 
The function $\mathcal{F}$ is annihilated by $\Delta_{\kappa}$.
\item
For every $\mathfrak{z}\in \H$, there exists $n\in\N_0$ such that $(z-\mathfrak{z})^n\mathcal{F}(z)$ is bounded in some neighborhood of $z$.
\item 
The function $\mathcal{F}$ grows at most linear exponentially towards cusps of $\Gamma$.
\end{enumerate}
Again, if one allows in condition (2) a general eigenvalue, then one obtains a polar Maass form.  One says that $\mathcal{F}$ has a \begin{it}singularity at $\mathfrak{z}\in\H$\end{it} if the minimal $n$ in condition (3) is positive.
\end{definition*}
We denote the space of polar harmonic Maass forms whose singularities in $\H$ are all poles and which are bounded towards all cusps by $\H_{\kappa}(\Gamma,\nu)$ and the subspace of meromorphic modular forms with no poles at the cusps by $\mathbb{M}_{\kappa}(\Gamma,\nu)$.  We furthermore use the notation $\SS_{\kappa}(\Gamma,\nu)$ for the subspace of {\it meromorphic cusp forms}, 
i.e., those meromorphic modular forms $f$ for which $y^{\frac{\kappa}{2}}|f(z)|$ decays towards all cusps.
 An important subspace of $\H_{\kappa}(\Gamma,\nu)$ is obtained by noting that the hyperbolic Laplacian splits as
\begin{equation}\label{eqn:Deltasplit}
\Delta_\kappa=-\xi_{2-\kappa}\circ \xi_\kappa,
\end{equation}
where $\xi_{\kappa,z}:=2iy^{\kappa} \overline{\frac{\partial}{\partial \overline{z}}}$.  
If $\mathcal{F}$ satisfies weight $\kappa$ modularity, then $\xi_{\kappa}(\mathcal{F})$ 
is modular of 
 weight $2-\kappa$ and one sees from the decomposition \eqref{eqn:Deltasplit} that $\xi_{\kappa}(\mathcal{F})\in \mathbb{M}_{2-\kappa}(\Gamma,\overline{\nu})$ if $\mathcal{F}\in \H_{\kappa}(\Gamma,\nu)$.  It is thus natural to consider the subspace $\H_{\kappa}^{\operatorname{cusp}}(\Gamma,\nu)\subseteq \H_{\kappa}(\Gamma,\nu)$ consisting of those $\mathcal{F}$ for which $\xi_{\kappa}(\mathcal{F})$ is a cusp form.

The Fourier expansion of a polar harmonic Maass form $\mathcal{F}$ splits naturally into meromorphic and non-meromorphic parts.  
Roughly speaking, the non-meromorphic part is the unique part of the Fourier expansion given in (not necessarily integral) powers of $e(\mathfrak{z})$ times the incomplete gamma function $\Gamma(\kappa,y):=\int_y^{\infty} e^{-t}t^{\kappa-1} dt$.  To be more precise, if $0\leq \ell_0<1$ satisfies $\ell_0-\ell \in\Z$, then the \begin{it}non-meromorphic part\end{it} has the shape 
\begin{equation*}
\mathcal{F}^-(z)=\sum_{\substack{n\gg -\infty\\n\neq \ell_0}}c^-(n)\Gamma\left(1-\kappa,\frac{4\pi \left(n-\ell_0\right)y}{L} \right) e^{-\frac{2\pi i \left(n-\ell_0\right)z}{L}}+\delta_{\ell_0=0} c^-(0)  y^{1-\kappa},
\end{equation*}
defined for $y$ sufficiently large, while the \begin{it}meromorphic part\end{it} $\mathcal{F}^+:=\mathcal{F}-\mathcal{F}^-$ can be written as $e(\ell_0 z /L)$ times an expansion in integral powers of $e(z/L)$.  These Fourier expansions hold whenever $y>\mathfrak{z}_2$ for every singularity $\mathfrak{z}\in\H$; in particular, if there is only one singularity $\mathfrak{z}$ modulo $\Gamma$, then the expansion holds whenever $y>\max(\mathfrak{z}_2,1/\mathfrak{z}_2)$.  Note that $\xi_{\kappa}(\mathcal{F})=\xi_{\kappa}(\mathcal{F}^-)$ implies that the coefficients $c^-(n)$ of the non-meromorphic part of $\mathcal{F}\in \H_{\kappa}^{\operatorname{cusp}}(\Gamma,\nu)$ are related to Fourier coefficients of a cusp form.  Specifically, for $\mathcal{F}\in \H_{\kappa}^{\operatorname{cusp}}(\Gamma,\nu)$, we have $c^-(n)=0$ for $n-\ell_0\leq 0$ and 
\begin{multline}\label{eqn:xiF}
\xi_{\kappa}\left( \sum_{n-\ell_0>0} c^-(n)\Gamma\left(1-\kappa,\frac{4\pi \left(n-\ell_0\right) y}{L}\right) e^{-\frac{2\pi i \left(n-\ell_0\right) z}{L}}\right)\\
 =-\left(\frac{4\pi}{L}\right)^{1-\kappa}\sum_{n-\ell_0>0}(n-\ell_0)^{1-\kappa}\overline{c^-(n)}e^{\frac{2\pi i(n-\ell_0) z}{L}}
\end{multline}
implies that $(n-\ell)^{1-\kappa}\overline{c^-(n)}$ are coefficients of a cusp form.  The coefficients of $\mathcal{F}^+$ are also related to coefficients of a weight $2-\kappa$ meromorphic modular form if $2-\kappa\in 2\N$ since $\mathcal{F}^-$ is annihilated by $D^{1-\kappa}$, where $D:=\frac{1}{2\pi i}\frac{\partial}{\partial z}$.  

For $\mathcal{F}\in H_{\kappa}$, we call the terms of the Fourier expansion which grow towards $i\infty$ the \begin{it}principal part\end{it} (at $i\infty$). In the same way we define principal parts at other cusps. For $\mathcal{F}\in \H_{\kappa}^{\operatorname{cusp}}$, the principal part at $i\infty$ is trivial, but there may be growth towards points $z\in \H$.  In particular, for $\mathcal{F}\in \H_{\kappa}^{\operatorname{cusp}}$ and $z\in\H$, there exist $c_1,\dots, c_{n_0}$ such that 
$$
\lim_{z\to \mathfrak{z}} \left(\mathcal{F}(z) - \sum_{n=1}^{n_0} c_n (z-\mathfrak{z})^{-n}\right)=O(1).
$$
We call $\sum_{n=1}^{n_0} c_n (\mathfrak{z}-z)^{-n}$ the \begin{it}principal part of $\mathcal{F}$ at $z$.\end{it}

\subsection{Poincar\'e series}\label{sec:Poincare}

An important tool to construct automorphic forms are \begin{it}Poincar\'e series\end{it}, which have a long history going back to Poincar\'e \cite{PoincareSeries}.  To be more precise, for $\kappa>2$ and $m\in\Z$, the classical weakly holomorphic Poincar\'e series are defined by (see (2b.1) of \cite{Pe1})
\begin{equation}\label{eqn:Pdef}
P_{\kappa,\ell+m,\Gamma,\nu}(z):={\displaystyle{\sum_{M\in \Gamma_{\infty}\backslash\Gamma}}} e^{\frac{2\pi i(\ell+m)z}{L}}\bigg|_{\kappa,\nu}M\in M_{\kappa}^!(\Gamma,\nu),
\end{equation}
where $\Gamma_{\infty}$ is the subgroup of $\Gamma$ generated by $T^L$. To construct harmonic Maass forms, for $\kappa<0$ and $m\in \Z$, we define the harmonic function 
$$
\varphi_{m}(z):=\mathcal{M}_{\kappa,1-\frac{\kappa}{2}}\left(\frac{4\pi (m+\ell_0) y}{L}\right)e^{\frac{2\pi i (m+\ell_0)x}{L}}
$$
with $(w\in\R\setminus\{0\},s\in\C)$ 
$$
\mathcal{M}_{\kappa,s}(w):=|w|^{-\frac{\kappa}{2}}M_{\frac{\kappa}{2}\sgn(w),s-\frac{1}{2}}(|w|),
$$
where $M_{\mu, \nu}$ is the $M$-Whittaker function.  This follows a construction of Niebur \cite{Niebur}, who considered eigenfunctions but restricted to weight 0.  In this paper, we only require the corresponding Poincar\'e series for integral weight $\kappa=2-2k$ with $k\in \N_{
>1
}$, trivial multiplier, and $\Gamma=\SL_2(\Z)$, so for the remainder of this section, we restrict to that case.  We define 
\begin{equation}\label{eqn:calFdef}
\mathcal{F}_{2-2k,m}:=-\frac{
\sgn(m)
}{(2k-1)!}\sum_{M\in\Gamma_{\infty}\backslash\SL_2(\Z)} 
\varphi_{m}
\bigg|_{2-2k}M\in H_{2-2k}.
\end{equation}
For $
m<0
$, the functions $\mathcal{F}_{2-2k,m}$ have principal parts $
2e(mz)
$ at $i\infty$. Furthermore, a straightforward calculation, using that $\xi$ commutes with the slash action, yields 
\begin{equation}\label{eqn:xiFP}
\xi_{2-2k}\left(\mathcal{F}_{2-2k,m}\right) = -\frac{(4\pi m)^{2k-1}}{(2k-2)!}P_{2k,-m}.
\end{equation}
This gives a relation between the $n$th Fourier coefficient $a_{2k,-m}(n)$ of $P_{2k,-m}$ and the $n$th coefficient $c_{2-2k,m}^-(n)$ of $\mathcal{F}_{2-2k,m}^{-}$. Comparing the coefficients on both sides of \eqref{eqn:xiFP}, using \eqref{eqn:xiF} to compute the coefficients of the left-hand side,
 we conclude that 
\begin{equation}\label{eqn:cnan}
c_{2-2k,m}^-(n) = \frac{1}{(2k-2)!}\left(\frac{m}{n}\right)^{2k-1}\overline{a_{2k,-m}(n)}.
\end{equation}
An inspection of the Fourier expansion yields that $a_{2k,-m}(n)$ is real. 
We therefore obtain 
\begin{equation}\label{eqn:duality}
c_{2-2k,m}^-(n) =\frac{1}{(2k-2)!}\left(\frac{m}{n}\right)^{2k-1}a_{2k,-m}(n)=\sgn(m)\frac{a_{2k,n}(-m)}{(2k-2)!}.
\end{equation}
where the second identity is so-called reciprocity, which follows by an argument similar to (5c.2) of \cite{Pe1}. Further similar results are contained in Theorem 2 of \cite{DIT} and follow by Theorem 3.7 of \cite{Fay}.  
It turns out that the $n$th coefficient $c_{2-2k,m}^+(n)$ of $\mathcal{F}_{2-2k,m}^+$ also satisfies another relation known as duality (cf. Theorem 1 of \cite{GuerzhoyGrids}), given in this case by 
\begin{equation}\label{eqn:dualityMaass}
c_{2-2k,m}^+(n)=\sgn(m)a_{2k,-n}(-m).
\end{equation}
A number of duality results have been of recent interest to many authors including Zagier \cite{ZagierTraces}, Ahlgren--Kim \cite{AhlgrenKim}, Duke--Jenkins \cite{DJ}, and Miller--Pixton \cite{MP}.  The correlations of Fourier coefficients given in \eqref{eqn:cnan} and \eqref{eqn:dualityMaass} play an important role in finding a nice formula for $y^{\kappa-1}\Psi_{\kappa}(\mathfrak{z},z)$ in the integral weight case.

\section{Polar harmonic Maass forms of Petersson}\label{sec:conj}

As noted in the introduction, the proof of Theorem \ref{Fouriercoefficients} relies on certain polar harmonic Maass forms defined by Petersson.  Although polar harmonic Maass forms were not investigated in Petersson's time and hence he did not note the harmonic properties of his function, he recognized the modularity properties.  Although we only need integral weight forms, Petersson asked about the general role that these functions play in real weight, so we introduce them here in that generality.  

To explicitly define the required functions, we first need some notation.  Following (5b.7) of \cite{Pe1}, for
 $\kappa\in\R_{>2}$, a congruence subgroup $\Gamma$ of $\SL_2(\Z)$ containing $T^L$ for some $L\in \N$ and $T:=\left(\begin{smallmatrix}1&1\\0&1\end{smallmatrix}\right)$, and a multiplier system $\nu$ satisfying $|\nu|=1$ and $\nu(T^{L})=e^{2\pi i\ell}$ with $0< \ell\leq  1$, define
\begin{equation}\label{eqn:Kdef}
K_{\kappa,\Gamma,\nu}\left(\mathfrak{z},z\right):=\frac{iL}{2\pi}(2iy)^{\kappa-1}\sum_{\substack{n\in\Z\\ n-\ell<0}} \mathcal{I}_{\kappa,n-\ell}\left(\frac{y}{L}\right)e^{\frac{2\pi i\left(n-\ell\right)z}{L}}P_{\kappa,\ell-n,\Gamma,\nu}(\mathfrak{z}),
\end{equation}
where $P_{\kappa,m,\Gamma,\nu}(\mathfrak{z})$ is the usual $m$th Poincar\'e series, $\mathcal{I}_{\kappa,r}(y)$ is a certain non-holomorphic integral defined in \eqref{eqn:Idef}, $\mathfrak{z}\in\H$, and $z=x+iy\in\H$ throughout.

Placing it into the theory of polar harmonic Maass forms yields a general formula for $K_{\kappa,\Gamma,\nu}$.  In order to give this, we define
$$
\Omega_{\kappa,\Gamma,\nu}(\mathfrak{z},z):=\sum_{M\in \Gamma} \frac{1}{\left(z-\overline{\mathfrak{z}}\right)^{\kappa}}\bigg|_{\kappa,\nu,z} M.
$$
\begin{theorem}\label{ConjectureTheorem}
For every $2<\kappa\in \R$, we have  
\begin{equation}\label{eqn:conj}
K_{\kappa,\Gamma,\nu}\left(\mathfrak{z},z\right)=\frac{i(\kappa-1)L}{2\pi}e^{-\pi i \kappa} \int_{-\overline{z}}^{i\infty}  \Omega_{\kappa, \Gamma, \overline{\nu}}\left( \overline{\mathfrak{z}}, -w\right)\left(z+w\right)^{\kappa-2} dw.
\end{equation}
\end{theorem}

\begin{remarks}
\noindent

\noindent
\begin{enumerate}[leftmargin=*]
\item 
Petersson \cite{Pe1} gave an alternate formula for $K_{\kappa,\Gamma,\nu}$ for $\kappa\in \N$, but noted that his formula does not extend to arbitrary real weight.  Although Theorem \ref{ConjectureTheorem} does not match his identity in integral weight, it places $K_{\kappa,\Gamma,\nu}$ into a larger framework.  

\item If $k\in \N$, then one can also use \eqref{eqn:xiH2} to rewrite $K_{2k}$ as 
$$
\frac{1}{(2k-2)!}\int_{-\overline{z}}^{i\infty} \frac{\partial^{2k-1}}{\partial w^{2k-1}} \left(H_{2k}\left(\mathfrak{z},-w\right)\right)\left(z+w\right)^{2k-2} dw,
$$
where here and throughout we omit the subgroup and multiplier in the notation whenever $\Gamma=\SL_2(\Z)$ and the multiplier is trivial.
\item
The Poincar\'e series $z\mapsto \Omega_{\kappa,\Gamma,\nu}(\mathfrak{z},z)$ were extensively studied in \cite{Pe4} and shown to be cusp forms.  Their analytic continuation to weight $2$ was the main goal in \cite{Pe3}.
\item  
For a weight $\kappa$ cusp form $f$, the integral
$$
(2i)^{1-\kappa}\int_{-\overline{z}}^{i\infty} \overline{f(-\overline{w})}\left(z+w\right)^{\kappa-2}dw
$$
is called the \begin{it}non-holomorphic Eichler integral\end{it} of $f$.  Hence Theorem \ref{ConjectureTheorem} states that $K_{\kappa,\Gamma,\nu}$ is (up to a non-zero constant) the non-holomorphic Eichler integral of the Poincar\'e series $\Omega_{\kappa,\Gamma,\nu}$.
\end{enumerate}

\end{remarks}

\subsection{Petersson's Poincar\'e series}
The main goal of this section is to realize the function $K_{\kappa,\Gamma,\nu}$ as the non-meromorphic part of a polar harmonic Maass form in order to prove Theorem \ref{ConjectureTheorem}.  To fill in the remaining piece of definition \eqref{eqn:Kdef}, for $r<0$ and $-2y<\alpha<0$, we define 
(see (5b.5) of \cite{Pe1}) 
\begin{equation}\label{eqn:Idef}
\mathcal{I}_{\kappa,r}(y):=\int_{i\alpha -\infty}^{i\alpha +\infty} \frac{e^{2\pi i r t}}{\left(t+2iy\right)^{\kappa-1}} \frac{dt}{t}.
\end{equation}

One of the main steps in proving Theorem \ref{ConjectureTheorem} is to connect other functions of Petersson to the theory of harmonic Maass forms.  These functions are defined by (here we renormalize $H_{\kappa,\Gamma,\nu}$ by multliplying (5b.3) of \cite{Pe1} by $\frac{L}{2\pi i}$) 
\begin{align}
\label{eqn:Hdef}
H_{\kappa,\Gamma,\nu}(\mathfrak{z},z) &:= \sum_{M \in \Gamma_{\infty}\backslash \Gamma} \frac{e^{\frac{2\pi i}{L} (1-\ell)(z-\mathfrak{z})}}{1-e^{\frac{2 \pi i}{L} (z-\mathfrak{z})}}\Bigg|_{\kappa,\nu,\mathfrak{z}} M,\\
\label{eqn:Psidef} \Psi_{\kappa,\Gamma,\nu}(\mathfrak{z},z)&:=\sum_{M\in \Gamma} \frac{1}{(\mathfrak{z}-z)\left(\mathfrak{z}-\overline{z}\right)^{\kappa-1}}\bigg|_{\kappa,\nu,\mathfrak{z}} M.
\end{align}
In particular, the functions $y^{\kappa-1}\Psi_{\kappa,\Gamma,\nu}(\mathfrak{z},z)$ are polar harmonic Maass forms.

\begin{proposition}\label{prop:HkPsi}
Assume that $\kappa$, $\Gamma\subseteq\SL_2(\Z)$, and the multiplier system $\nu$ satisfy the conditions given above.  Then $z\mapsto y^{\kappa-1}\Psi_{\kappa,\Gamma,\nu}(\mathfrak{z},z)
\in \H_{2-\kappa}^{\operatorname{cusp}}\left(\Gamma,\overline{\nu}\right)$.  Its meromorphic part equals
$$
\frac{2\pi i}{(2i)^{\kappa-1}L}H_{\kappa,\Gamma,\nu}(\mathfrak{z},z).
$$
\end{proposition}
\begin{remarks}
\noindent

\noindent
\begin{enumerate}[leftmargin=*]
\item
In (5b.6) of \cite{Pe1}, the Poincar\'e series 
\eqref{eqn:Kdef}, \eqref{eqn:Hdef}, and \eqref{eqn:Psidef} are related by  the decomposition (recall that we have renormalized $H_{\kappa, \Gamma,\nu}$ by dividing by $\frac{2\pi i}{L}$) 
\begin{equation}\label{eqn:split}
(2iy)^{\kappa-1}\Psi_{\kappa,
\Gamma,\nu
}(\mathfrak{z},z)=\frac{2\pi i}{L} H_{\kappa,\Gamma,\nu}(\mathfrak{z},z) -\frac{2\pi i}{L}K_{\kappa,\Gamma,\nu}(\mathfrak{z},z).
\end{equation}
By Proposition \ref{prop:HkPsi}, \eqref{eqn:split} is a natural splitting into meromorphic and non-meromorphic parts.  Note also that by \eqref{eqn:split}, $K_{\kappa,\Gamma,\nu}$ has a representation as a Poincar\'e series (after rewriting $\Psi_{\kappa,\Gamma,\nu}$ as a Poincar\'e series over $\Gamma_{\infty}\backslash\Gamma$).
\item
Petersson defined $\Psi_{\kappa,\Gamma,\nu}$ for Fuchsian groups of the first kind which contain $T^L$.  We restrict to congruence subgroups to simplify the development of the theory of polar harmonic Maass forms, but the proof of Proposition \ref{prop:HkPsi} is formal, so one should be able to prove a version of Proposition \ref{prop:HkPsi} in the generality of the full definition.
\item  The functions $H_{\kappa,\Gamma,\nu}$ are a special case of a more general family considered in \cite{PeTheorie}. In particular, $H_{\kappa,\Gamma,\nu}$ is $e^{2\pi i (1-\ell) z/L}$ times the function defined in (31) of \cite{PeTheorie} with $F(t)=t^{m}(t-e^{2\pi i z/L})^{-1}$, where $m=0$ if $\ell\neq 1$ and $m=1$ if $\ell=1$. 
\end{enumerate}
\end{remarks}
\begin{proof}[Proof of Proposition \ref{prop:HkPsi}] 
From condition 5 on page 63 of \cite{Pe1}, we conclude that $z\mapsto y^{\kappa-1} \Psi_{\kappa,\Gamma,\nu}(\mathfrak{z},z)$ transforms like a modular form of weight $2-\kappa$ with multiplier $\overline{\nu}$.  Furthermore, we compute
\begin{equation*}
\xi_{2-\kappa,z}\left(y^{\kappa-1}\left(\mathfrak{z}-\overline{z}\right)^{1-\kappa}\left(\mathfrak{z}-z\right)^{-1}\right)=(\kappa-1)\left(\overline{\mathfrak{z}}-z\right)^{-\kappa}, \quad \text{ and hence}\end{equation*}
$$
\xi_{2-\kappa,z}\left(y^{\kappa-1}\Psi_{\kappa,\Gamma, \nu}(\mathfrak{z},z)\right) = (\kappa-1)\sum_{M\in\Gamma}\overline{\nu}(M)^{-1} j(M,\overline{\mathfrak{z}})^{-\kappa} \left(M\overline{\mathfrak{z}}-z\right)^{-\kappa}.
$$

We now use (23) of \cite{Pe2}, followed by (2a.4) of \cite{Pe1}, to rewrite the sum over $M\in\Gamma$ as 
\begin{equation}\label{eqn:slashz}
\sum_{M\in\Gamma} \nu(M)j\left(M,M^{-1}z\right)^{\kappa}\left(\overline{\mathfrak{z}}-M^{-1}z\right)^{-\kappa}=\sum_{M\in\Gamma} \nu\left(M^{-1}\right)^{-1}j(M^{-1},z)^{-\kappa}\left(\overline{\mathfrak{z}}-M^{-1}z\right)^{-\kappa}.
\end{equation}
Using that $(\overline{\mathfrak{z}}-M^{-1}z)^{-\kappa}=e^{\pi i \kappa}(M^{-1}z-\overline{\mathfrak{z}})^{-\kappa}$, \eqref{eqn:slashz} implies that 
\begin{equation}\label{eqn:xi}
\xi_{2-\kappa,z}\left(y^{\kappa-1}\Psi_{\kappa,\Gamma, \nu}(\mathfrak{z},z)\right) = (\kappa-1)e^{\pi i \kappa}\sum_{M\in\Gamma} \left(z-\overline{\mathfrak{z}}\right)^{-\kappa}\bigg|_{\kappa,\nu,z}M=(\kappa-1)e^{\pi i \kappa}
\Omega
_{\kappa,\Gamma, \nu}(\mathfrak{z},z),
\end{equation}
which is a cusp form in $z$.

Using \eqref{eqn:Deltasplit}, we conclude that $
z\mapsto 
y^{\kappa-1}\Psi_{\kappa,\Gamma,\nu}(\mathfrak{z},z)\in \H_{2-\kappa}^{\operatorname{cusp}}(\Gamma,\overline{\nu})$.  Noting that $
z \mapsto 
H_{\kappa,\Gamma,\nu}
(\mathfrak{z},z)
$ is meromorphic, and hence in particular harmonic, \eqref{eqn:split} implies that $K_{\kappa,\Gamma,\nu} (\mathfrak{z}, z)$ is also harmonic. 
Moreover, 
as $y\to \infty$, 
$K_{\kappa,\Gamma,\nu} (\mathfrak{z}, z)$ decays by \eqref{eqn:split} and the fact that  \eqref{eqn:Hdef} and $y^{\kappa-1}$ times \eqref{eqn:Psidef} vanish as $y \to \infty$.  We conclude that there exist $b_{n}=b_n(\mathfrak{z}),\ a_{n}=a_n(\mathfrak{z})\in \C$ such that
$$
K_{\kappa,\Gamma,\nu}(\mathfrak{z},z) = \sum_{n-\ell_0 \geq 0}b_{n} e^{\frac{2\pi i}{L} \left(n-\ell_0\right) z} + \sum_{n-\ell_0>0} a_{n} \Gamma\left(\kappa-1,\frac{4\pi \left(n-\ell_0\right)y}{L}  \right) e^{-\frac{2\pi i}{L}\left(n-\ell_0\right) z}.
$$
By \eqref{eqn:Kdef}, the coefficients are only supported on negative powers of $e(x)$ and hence $K_{\kappa, \Gamma, \nu}$ is an expansion in incomplete gamma functions.  Since $
z\mapsto 
H_{\kappa,\Gamma,\nu}
(\mathfrak{z},z)
$ is meromorphic, we conclude that $K_{\kappa,\Gamma,\nu}$ is precisely the non-meromorphic part of $\frac{iL}{2\pi}(2iy)^{\kappa-1}\Psi_{\kappa,\Gamma,\nu}(\mathfrak{z},z)$, giving the claim. 
\end{proof}
Proposition \ref{prop:HkPsi} is the main step needed to prove Theorem \ref{ConjectureTheorem}.
\begin{proof}[Proof of Theorem \ref{ConjectureTheorem}]
By Proposition \ref{prop:HkPsi}, $K_{\kappa,\Gamma,\nu} (\mathfrak{z}, z)$ is the non-meromorphic part of the function $\frac{i(2i)^{\kappa-1}L}{2\pi}y^{\kappa-1}\Psi_{\kappa,\Gamma,\nu}(\mathfrak{z},z)$. An easy calculation gives that for a weight $\kappa$ cusp form $f$, we have 
$$
\xi_{2-\kappa}\left((2i)^{1-\kappa}\int_{-\overline{z}}^{i\infty} \overline{f(-\overline{w})}\left(z+w\right)^{\kappa-2}dw\right) = f(z).
$$
By \eqref{eqn:xi}, we conclude that both sides of \eqref{eqn:conj} have the same image under $\xi_{2-\kappa,z}$.  Moreover, each side has an expansion in incomplete gamma functions, and such 
a series is annihilated by $\xi_{2-\kappa,z}$ if and only if the function is identically zero, yielding the claim.
\end{proof}

\subsection{The case of integral weight}
From now on we assume even integral weight, $\Gamma=\SL_2(\Z)$, and trivial multiplier.  We consider the polar harmonic  Maass form completion of $H_{2k}$ in another light and find an even simpler representation for it.  For this, define 
\begin{equation}\label{eqn:HtauHtaubar}
\mathcal{H}_{2k}\left(\mathfrak{z},z\right) := H_{2k}(\mathfrak{z},z) + \sum_{r=0}^{2k-2}\frac{(2iy)^r }{r!}\frac{\partial^r}{\partial\overline{z}^r} H_{2k}\left(\mathfrak{z},\overline{z}\right).
\end{equation}
Note that although we assume $z\in \H$ in the definition \eqref{eqn:Hdef} of $H_{2k}$, this function is well-defined for all $z\in\C$, and hence \eqref{eqn:HtauHtaubar} is meaningful.

\begin{proposition}\label{prop:modulartau}
We have $\mathfrak{z}\mapsto \mathcal{H}_{2k}(\mathfrak{z},z)\in \mathbb{M}_{2k}$ and $z\mapsto\mathcal{H}_{2k}(\mathfrak{z},z)\in \H_{2-2k}^{\operatorname{cusp}}$.  Furthermore
\begin{equation}\label{eqn:HkPsi}
\mathcal{H}_{2k}(\mathfrak{z},z) =-\frac{i(2i)^{2k-1}}{2\pi}  y^{2k-1}\Psi_{2k}(\mathfrak{z},z).
\end{equation}
\end{proposition}
\begin{proof} 
Petersson \cite{Pe1} showed that $\mathfrak{z}\mapsto H_{2k}(\mathfrak{z},z)\in \mathbb{M}_{2k}$, while the analogous result for $H_{2k}(\mathfrak{z},\overline{z})$ and its derivatives follows mutatis mutandis, yielding the assertion for $\mathfrak{z}$.

To show the statement in $z$, we rewrite $\mathcal{H}_{2k}$, for $\mathfrak{z}_2>\max(y,1/y)$, as the generating function 
\begin{equation}\label{eqn:Hkdef}
\mathcal{H}_{2k}\left(\mathfrak{z},z\right)=-\sum_{n=1}^{\infty}\mathcal{F}_{2-2k,-n}(z) e^{2\pi i n\mathfrak{z}}
\end{equation}
of the (weight $2-2k$) Maass--Poincar\'e series $\mathcal{F}_{2-2k,-n}$ defined in \eqref{eqn:calFdef}.  We begin proving \eqref{eqn:Hkdef} by noting that, by (3a.8) of \cite{Pe1}, we have for $\mathfrak{z}_2>\max(y,1/y)$ 
\begin{equation*}
H_{2k}\left(\mathfrak{z},z\right)=-\sum_{n=1}^{\infty}F_{2-2k,-n}\left(z\right) e^{2\pi i n\mathfrak{z}},
\end{equation*}
where (recalling that $a_{2k,-j}(n)$ is the $n$th coefficient of $P_{2k,-j}$) 
\begin{equation*}
F_{2-2k,-n}(z):= 2e^{-2\pi i nz} -\sum_{j=0}^{\infty}a_{2k,-j}(n)e^{2\pi i j z}.
\end{equation*}
By \eqref{eqn:dualityMaass}, one concludes that $F_{2-2k,-n}=\mathcal{F}_{2-2k,-n}^+$.  It remains to realize the remainder of $\mathcal{H}_{2k}$ as the generating function for $-\mathcal{F}_{2-2k,-n}^-$.  Noting that $\im(M\mathfrak{z}-\overline{z})>0$ for all $\mathfrak{z},z\in\H$, and $M\in\SL_2 (\Z)$ we expand the geometric series in the definition of $H_{2k}$ to rewrite 
\begin{equation*}
G_{2k}(\mathfrak{z},z) :=\sum_{r=0}^{2k-2}\frac{\left(2i y\right)^{r}}{r!} \frac{\partial^r}{\partial\overline{z}^r}H_{2k}\left(\mathfrak{z},\overline{z}\right)=-\sum_{r=0}^{2k-2}\frac{\left(4\pi y\right)^{r}}{r!} \sum_{M\in \Gamma_{\infty}\backslash \SL_2(\Z)} \sum_{j= 1}^{\infty} j^{r}e^{2\pi i j\left(\mathfrak{z}-\overline{z}\right)}\bigg|_{2k,\mathfrak{z}} M.
\end{equation*}
We then use that
$$
(2k-2)!e^{-4\pi j y}\sum_{r=0}^{2k-2}\frac{\left(4\pi jy\right)^{r}}{r!}=\Gamma\left(2k-1,4\pi j y\right)
$$
to obtain, plugging in the Fourier expansion of $P_{2k, j}$, 
$$
G_{2k}(\mathfrak{z},z) = -\frac{1}{(2k-2)!}\sum_{j=1}^\infty\Gamma\left(2k-1,4\pi j y\right) e^{-2\pi i j z} \sum_{n=1}^{\infty}a_{2k,j}(n)e^{2\pi i n\mathfrak{z}}.
$$
Reordering the sums and inserting \eqref{eqn:duality} then yields
\begin{equation*}
G_{2k}(\mathfrak{z},z)=-\sum_{n=1}^{\infty}\mathcal{F}_{2-2k,-n}^-\left(z\right) e^{2\pi i n\mathfrak{z}},
\end{equation*}
verifying \eqref{eqn:Hkdef}.

The weight $2-2k$ modularity of $\mathcal{F}_{2-2k,-n}$ then implies that $\mathcal{H}_{2k}(\mathfrak{z},z)|_{2-2k,z} M -\mathcal{H}_{2k}(\mathfrak{z},z)=0$ 
if $\mathfrak{z}_2>\max(y,1/y)$.  Since $\mathfrak{z}\mapsto\mathcal{H}_{2k}(\mathfrak{z},z)$, and hence $\mathfrak{z}\mapsto\mathcal{H}_{2k}(\mathfrak{z},z)|_{2-2k,z} M -\mathcal{H}_{2k}(\mathfrak{z},z)$, is meromorphic, we conclude by the identity theorem that $z\mapsto\mathcal{H}_{2k}(\mathfrak{z},z)$ is modular.  Furthermore, \eqref{eqn:Hkdef} implies that $z\mapsto\mathcal{H}_{2k}(\mathfrak{z},z)$ is harmonic whenever $\mathfrak{z}_2>\max(y,1/y)$.  Since $\mathfrak{z}\mapsto\Delta_{2-2k,z}\left(\mathcal{H}_{2k}(\mathfrak{z},z)\right)$ is meromorphic, the identity theorem implies that $z\mapsto\mathcal{H}_{2k}(\mathfrak{z},z)$ is harmonic.  To verify that $\mathcal{H}_{2k}\in \mathbb{H}_{2-2k}^{\text{cusp}}$, we act by $\xi_{2-2k,z}$ on \eqref{eqn:HtauHtaubar}, noting that 
$$
\overline{\frac{\partial}{\partial \overline{z}} H_{2k}(\mathfrak{z},\overline{z})}=\frac{\partial}{\partial z} H_{2k}(-\overline{\mathfrak{z}},-z).
$$

A straightforward calculation using telescoping series 
then yields
\begin{equation}\label{eqn:xiH2}
\xi_{2-2k,z}\left(\mathcal{H}_{2k}(\mathfrak{z},z)\right)=-\frac{(2i)^{2k-1}}{(2k-2)!} \frac{\partial^{2k-1}}{\partial z^{2k-1}} H_{2k}\left(-\overline{\mathfrak{z}},-z\right).
\end{equation}
The right-hand side of \eqref{eqn:xiH2} clearly decays towards $i\infty$ and is hence  a cusp form because it has no poles in $\H$.

Taking the difference of both sides of \eqref{eqn:HkPsi}, one obtains a harmonic Maass form in $z$.  Moreover, both sides are bounded towards $i\infty$ and their meromorphic parts match by \eqref{eqn:HtauHtaubar} and Proposition \ref{prop:HkPsi} so that the difference has no poles in $\H$.  Since its weight is negative, the difference is zero, which yields \eqref{eqn:HkPsi}.
\end{proof}

\section{Proof of Theorem \ref{Fouriercoefficients}}\label{sec:Fouriercoeff}

\subsection{A generalization of Theorem \ref{Fouriercoefficients}}
We first give the definitions necessary to state a more general version of Theorem \ref{Fouriercoefficients}.  
For $(c,d)=1$, we define 
\begin{equation*}
B_{k,c,d}(\mathfrak{z},n):=\frac{1}{\left(c\mathfrak{z}+d\right)^k}e^{\frac{2\pi n \mathfrak{z}_2}{\left|c\mathfrak{z}+d\right|^2}}e\left(-\frac{n}{\left|c\mathfrak{z}+d\right|^2}\left(ac|\mathfrak{z}|^2 + bd +u\left(ad+bc\right)\right)\right).
\end{equation*}
The Fourier expansions of negative-weight meromorphic cusp forms are given by the following theorem.
\begin{theorem}\label{thm:FourierCoeff}
If $f\in \SS_{2-2k}$ is written in the form
\begin{equation}\label{eqn:meroRexp}
f(z) = \sum_{l=1}^r \sum_{n=0}^{n_l} a_{l,n} R_{2k,\mathfrak{z}}^{n}\left[H_{2k}\left(\mathfrak{z},z\right)\right]_{\mathfrak{z}=\tau_{l}},
\end{equation}
then, for $m\in\N_0$, the $m$th Fourier coefficient of $f$ is given by 
$$
\sum_{\substack{c,d\in \Z\\ (c,d)=1}}  \sum_{l=1}^r \sum_{n=0}^{n_l} a_{l,n} \sum_{j=0}^{n} \frac{(2k+n-1)!}{(2k+n-1-j)!}\binom{n}{j} (-2i)^{n-j}\left(
\frac{\left|c\tau_l+d\right|^2}{v_{l}}\right)^{j}\left(2\pi i m\right)^{n-j} B_{2k+2n,c,d}\left(\tau_{l},m\right).
$$
\end{theorem}
\begin{remark}

In Proposition \ref{prop:basis} below, we see that the meromorphic parts of elements of 
$\H_{2-2k}^{\operatorname{cusp}}$ have
 the shape \eqref{eqn:meroRexp} and hence also have Fourier expanions as given in Theorem \ref{thm:FourierCoeff}.
\end{remark}

Before proving Theorem \ref{thm:FourierCoeff} in Subsection \ref{sectionprooftheorem},  we first show that all meromorphic cusp forms indeed have a representation as in \eqref{eqn:meroRexp}.  For this, we construct polar harmonic Maass forms with all possible principal parts.  Since the Laplace operator in the $z$ variable commutes with the raising operator in the $\mathfrak{z}$ variable, we see from \eqref{eqn:Hdef} that the functions $R_{2k,\mathfrak{z}}^n\left[\mathcal{H}_{2k}(\mathfrak{z},z)\right]_{\mathfrak{z}=\tau_{l}}$ are all polar harmonic  Maass forms.  These turn out to span the space of such forms.
\begin{proposition}\label{prop:basis}
Every $f\in \H_{2-2k}^{\operatorname{cusp}}$ has a presentation of the form 
\begin{equation}\label{eqn:meroRcalexp}
f(z) = \sum_{l=1}^r \sum_{n=0}^{n_l} a_{l,n} R_{2k,\mathfrak{z}}^{n}\left[\mathcal{H}_{2k}\left(\mathfrak{z},z\right)\right]_{\mathfrak{z}=\tau_{l}}.
\end{equation}
Furthermore, for $f$ satisfying \eqref{eqn:meroRcalexp}, we have $f\in \SS_{2-2k}$ if and only if \eqref{eqn:meroRexp} holds.
\end{proposition}
To prove Proposition \ref{prop:basis}, we determine the principal 
part of $R_{2k,\mathfrak{z}}^{n}\left[\mathcal{H}_{2k}(\mathfrak{z},z)\right]_{\mathfrak{z}=\tau_{l}}$. Since the non-meromorphic part of $\mathcal{H}_{2k}(\mathfrak{z},z)$ has no singularities, and hence does not contribute to the principal part, this is equivalent to determining the principal parts of each $R_{2k,\mathfrak{z}}^{n}\left[H_{2k}(\mathfrak{z},z)\right]_{\mathfrak{z}=\tau_{l}}$.  
We hence fix $\tau_0\in\H$ and compute the corresponding principal part at $z=\tau_0$ for $\tau_l=\tau_0$, observing that the principal parts at points inequivalent to $\tau_0$ are trivial.  
Recalling that we have renormalized $H_{2k}(\tau_0,z)$, we begin by noting that Petersson (see (3a.10) of \cite{Pe1}) has shown that its principal part is
\begin{equation*}
\frac{\widetilde{\varepsilon}_{2k}\left(\tau_0\right)}{z-\tau_0}, \quad \text{where}
\end{equation*} 
\begin{equation}\label{eqn:Resztau}
\widetilde{\varepsilon}_{2k}\left(\tau_0\right):=
\begin{cases}
\frac{i}{2\pi}\omega_{\tau_0}& \text{if }k\equiv 0 \pmod{\omega_{\tau_0}},\\
0&\text{otherwise}.
\end{cases}
\end{equation}

The fact that the raising and slash operators commute yields the following lemma.
\begin{lemma}\label{lem:Rprincpart}
The principal part of $R_{2k,\mathfrak{z}}^n\left[H_{2k}(\mathfrak{z},z)\right]_{\mathfrak{z}=\tau_0}$ at $z=\tau_0$ equals
$$
\widetilde{\varepsilon}_{2k+2n}\left(\tau_0\right)R_{2k,\mathfrak{z}}^n\left[ \frac{1}{z-\mathfrak{z}}\right]_{\mathfrak{z}=\tau_0}.
$$
\end{lemma}
To use Lemma \ref{lem:Rprincpart} for specific examples, we explicitly compute the principal part in Lemma \ref{lem:Rprincpartexplicit}.  
Although by Proposition \ref{prop:basis} all polar harmonic  Maass forms may be represented as a linear combination of the forms $R_{2k,\mathfrak{z}}^n\left[\mathcal{H}_{2k}(\mathfrak{z},z)\right]_{\mathfrak{z}=\tau_{l}}$, this representation does not have to be unique.  To find a basis of weight $2-2k$ polar harmonic Maass forms, we first determine the possible principal parts of such 
forms.
\begin{lemma}\label{lem:congruence}
For every $\tau_0\in\H$ and $n\in\N$, there exists an element of $\H_{2-2k}^{\operatorname{cusp}}$ with principal part 
\begin{equation}\label{eqn:ppharmonic}
\frac{1}{\left(z-\tau_0\right)^n}+O\left(\frac{1}{\left(z-\tau_0\right)^{n-1}}\right)
\end{equation}
around $z=\tau_0$ if and only if $n\equiv 1-k \pmod{
\omega_{\tau_0}}$.
\end{lemma}
\begin{proof}
First suppose that $g\in \H_{2-2k}^{\operatorname{cusp}}$ with principal part \eqref{eqn:ppharmonic} exists. Applying $D^{2k-1}$ yields
$$
D^{2k-1}\left(g(z)\right) = \frac{(n+2k-2)!}{\left(-2\pi i\right)^{2k-1}(n-1)! \left(z-\tau_0\right)^{n+2k-1}} + O\left(\frac{1}{\left(z-\tau_0\right)^{n+2k-2}}\right).
$$
However, by Bol's identity \cite{Bol}, $D^{2k-1}(g)=
(-4\pi)^{1-2k}
R_{2-2k}^{2k-1}(g)$ is a meromorphic modular form of weight $2k$.  Recall that for every weight $2k$ meromorphic modular form $f$ there exist complex numbers $b_{m}\left(\tau_0\right)$ (cf. (2a.16) of \cite{Pe1}) for which
\begin{equation}\label{eqn:ellexp}
f(z) = \left(z-\overline{\tau_0}\right)^{-2k} \sum_{\substack{m\gg -\infty\\ m\equiv -k\pmod{
\omega_{\tau_0}}}} b_{m}
\left(\tau_0\right)\left(\frac{z-\tau_0}{z-\overline{\tau_0}}\right)^m.
\end{equation}
 Hence by the congruence condition in \eqref{eqn:ellexp}, 
we have
 $n\equiv 1-k\pmod{\omega_{\tau_0}
}$.  

Conversely, if $n\equiv 1-k \pmod{\omega_{\tau_0}}$, then $\widetilde{\varepsilon}_{2k+2n-2}\left(\tau_0\right)\neq 0$, and hence a constant multiple of $R_{2k,\mathfrak{z}}^{n-1}\left[\mathcal{H}_{2k}(\mathfrak{z},z)\right]_{\mathfrak{z}=\tau_0}\in \H_{2-2k}^{\operatorname{cusp}}$ has the principal part \eqref{eqn:ppharmonic}. 
\end{proof}
We are now ready to prove Proposition \ref{prop:basis}.
\begin{proof}[Proof of Proposition \ref{prop:basis}]
For $g\in \H_{2-2k}^{\operatorname{cusp}}$, we recursively construct a polar harmonic Maass form with a lower-order pole in $
\tau_l
$.  Let $n+1$ be the order of the pole of $g$ at $\tau_{l}$. Lemma \ref{lem:congruence} implies that $n\equiv -k\pmod{\omega_{\tau_l}}$, while by Lemma \ref{lem:Rprincpart} we can subtract a constant multiple of $R_{2k,\mathfrak{z}}^{n}\left[\mathcal{H}_{2k}(\mathfrak{z},z)\right]_{\mathfrak{z}=\tau_{l}}$ to obtain a polar harmonic Maass form with a lower-order pole. Repeating this process, gives a harmonic Maass form with no singularities, yielding an identity because there are no non-trivial negative-weight harmonic Maass forms without any singularities.  Finally note that \eqref{eqn:meroRcalexp} is meromorphic if and only if the right-hand side equals its meromorphic part, which is \eqref{eqn:meroRexp}.
\end{proof}

\subsection{Proof of Theorem \ref{thm:FourierCoeff}}\label{sec:FourierCoeff}  
\label{sectionprooftheorem}

In order to obtain formulas for the Fourier coefficients of meromorphic cusp forms, we rewrite $R_{2k,\mathfrak{z}}^n\left(H_{2k}(\mathfrak{z},z)\right)$ as linear combinations of derivatives in $z$ of 
\begin{equation*}
H_{2k,j}(\mathfrak{z},z):=\sum_{M\in \Gamma_{\infty}\backslash \SL_2(\Z)}\frac{\mathfrak{z}_2^{-j}}{1-e^{2\pi i (z-\mathfrak{z})}}\bigg|_{2k,\mathfrak{z}} M,
\end{equation*}
whose Fourier coefficients were computed in Theorem 3.1 of \cite{BKFC} (see Theorems \ref{genFC} and \ref{genFC2} below).
\begin{proposition}\label{prop:Rtaugen}
For $n\in \N_0$, we have
\begin{equation}\label{eqn:Rtaugencoeff}
R_{2k,\mathfrak{z}}^{n}\left(H_{2k}(\mathfrak{z},z)\right)=\sum_{j=0}^{n} \frac{(2k+n-1)!}{(2k+n-1-j)!}\binom{n}{j} (-2i)^{n-j} \frac{\partial^{n-j}}{\partial z^{n-j}} H_{2k+2n,j}(\mathfrak{z},z).
\end{equation}
\end{proposition}
\begin{remark}
By taking the completions on both sides, we construct a polar harmonic  Maass form as
$$
R_{2k,\mathfrak{z}}^{n}\left(\mathcal{H}_{2k}(\mathfrak{z},z)\right)= \sum_{j=0}^{n} \frac{(2k+n-1)!}{(2k+n-1-j)!}\binom{n}{j} (-2i)^{n-j} \frac{\partial^{n-j}}{\partial z^{n-j}} \mathcal{H}_{2k+2n,j}(\mathfrak{z},z).
$$
\end{remark}
\begin{proof}[Proof of Proposition \ref{prop:Rtaugen}]
We prove (\ref{eqn:Rtaugencoeff}) by induction on $n$.  The case $n=0$ is trivial.  

Before applying the induction step, we determine the action of raising in $\mathfrak{z}$ on $H_{2k,j}$ to turn derivatives with respect to $\mathfrak{z}$ into derivatives with respect to $z$.  Commuting the raising operator with the slash operator, a direct calculation yields that
\begin{equation}\label{eqn:RtauHkj}
R_{2k,\mathfrak{z}}\left(H_{2k,j}(\mathfrak{z},z)\right) = -2i\frac{\partial}{\partial z} H_{2k+2,j}(\mathfrak{z},z) + (2k-j) H_{2k+2,j+1}(\mathfrak{z},z).
\end{equation}
Assuming now by induction that \eqref{eqn:Rtaugencoeff} holds for $n$ gives 
\begin{equation}\label{eqn:geninduct}
R_{2k,\mathfrak{z}}^{n+1}\left(H_{2k}(\mathfrak{z},z)\right) =\sum_{j=0}^{n} \frac{(2k+n-1)!}{(2k+n-1-j)!}\binom{n}{j} (-2i)^{n-j} \frac{\partial^{n-j}}{\partial z^{n-j}}R_{2k+2n,\mathfrak{z}}\left( H_{2k+2n,j}(\mathfrak{z},z) \right).
\end{equation}
We then plug in \eqref{eqn:RtauHkj} to rewrite \eqref{eqn:geninduct}, after a shift in the second term, as 
$$
\sum_{j=0}^{n+1}\frac{(2k+n-1)!(-2i)^{n+1-j}}{(2k+n-j)!}\frac{\partial^{n+1-j}}{\partial z^{n+1-j}}H_{2k+2n+2,j}(\mathfrak{z},z)\!\left(\!(2k+n-j)\binom{n}{j}\!+(2k+2n+1-j)\binom{n}{j-1}\!\right)\!.
$$
Using the identity
\begin{equation}\label{eqn:binom}
(2k+n-j)\binom{n}{j}+ (2k+2n+1-j)\binom{n}{j-1} = (2k+n)\binom{n+1}{j},
\end{equation}
then yields that \eqref{eqn:Rtaugencoeff} also holds for $n+1$, finishing the proof.
\end{proof}
As alluded to above, one of the key steps in proving Theorem \ref{thm:FourierCoeff} is to use Proposition \ref{prop:Rtaugen} together with Theorem 3.1 of \cite{BKFC} to compute the Fourier coefficients of $R_{2k,\mathfrak{z}}^{n}\left[H_{2k}(\mathfrak{z},z)\right]_{\mathfrak{z}=\tau_0}$.  For the convenience of the reader, before proving Theorem \ref{thm:FourierCoeff}, we recall Theorem 3.1 (1) of \cite{BKFC} (note that we have renormalized $H_{2k,j}$ by a factor of $2\pi i$).
\begin{theorem}\label{genFC} 
If $j,r\in \N_0$, $k\geq 2+j$, and $y>\im(M\mathfrak{z})$ holds for every $M\in \SL_2(\Z)$, then 
$$
\frac{\partial^r}{\partial z^r} H_{2k,j}\left(\mathfrak{z},z\right)=\sum_{m=0}^{\infty} \sum_{\substack{c,d\in \Z\\ (c,d)=1}} \left(\frac{\left|c\mathfrak{z}+d\right|^2}{\mathfrak{z}_2}\right)^{j}\left(2\pi i m\right)^r B_{2k,c,d}(\mathfrak{z},m) e^{2\pi i mz}.
$$
\end{theorem}

\begin{proof}[Proof of Theorem \ref{thm:FourierCoeff}]
We represent $f$ as in \eqref{eqn:meroRexp} and employ Proposition 
\ref{prop:Rtaugen} to write 
\begin{equation}\label{fasH}
f(z) = \sum_{l=1}^r \sum_{n=
0
}^{n_l} a_{l,n} \sum_{j=0}^{n} \frac{(2k+n-1)!}{(2k+n-1-j)!}\binom{n}{j} (-2i)^{n-j} \frac{\partial^{n-j}}{\partial z^{n-j}} H_{2k+2n,j}\left(\tau_{l},z\right).
\end{equation}
Theorem \ref{genFC} then yields the claim.
\end{proof}

As mentioned in the introduction, in the special case that $\tau_{l}=i$ or $\tau_{l}=\rho$, Theorem \ref{thm:FourierCoeff} has a particularly pleasing shape, which we next describe.  
We use the notation $\mathcal{O}_K$ for the ring of integers of a number field $K$ and write ideals of $\mathcal{O}_K$ as $\mathfrak{b}\subseteq \mathcal{O}_K$.  We call the ideals of $\mathfrak{b}\subseteq\mathcal{O}_K$ which are not divisible by any principal ideal $(g)$ with $g\in \Z$ \begin{it}primitive\end{it} and denote the sum over all primitive ideals of $\mathcal{O}_K$ by $\sum_{\mathfrak{b}\subseteq\mathcal{O}_K}^*$.  For $\gamma=c\mathfrak{z}+d\in \mathcal{O}_{\Q(\mathfrak{z})}$, we let $N(\gamma)$ be the norm of $\gamma$ in $\mathcal{O}_K$; we also use this notation for norms of ideals.  
For $\mathfrak{b}=(c\rho+d)\subset\mathcal{O}_{\Q(\rho)}$, we define 
\begin{equation}\label{eqn:C6}
C_{6m}\left(\mathfrak{b},n\right) := 
\cos\left(\frac{\pi n}{N(\mathfrak{b})}\left(ad+bc-2ac-2bd\right) -6m\arctan\left(\frac{c\sqrt{3}}{2d-c}\right)\right),
\end{equation}
where $a,b\in \Z$ are any choices for which $ad-bc=1$ and we set $C_{m}(\mathfrak{b},n):=0$ if $6\nmid m$.  Similarly, for $\mathfrak{b}=(ci+d)\subseteq\mathcal{O}_{\Q(i)}$, we let
\begin{equation}\label{eqn:C4}
C_{4m}\left(\mathfrak{b},n\right):=\cos\left(\frac{2\pi n}{N(\mathfrak{b})}\left(ac+bd\right)+4m\arctan\left(\frac{c}{d}\right)\right)
\end{equation}
and $C_{m}(\mathfrak{b},n):=0$ if $4\nmid m$.  Rewriting $B_{2k+2n,c,d}\left(\tau_{l},m\right)$ in terms of the $C_{2k}$ functions yields Theorem 3.1 (2) of \cite{BKFC}.
\begin{theorem}\label{genFC2}
For every $\mathfrak{z}\in \{i,\rho\}$, $j,r\in\N_0$, $k\geq 2+j$, and $y>\mathfrak{z}_2$, one has 
$$
\frac{\partial^r}{\partial z^r}H_{2k,j}\left(\mathfrak{z},z\right) =2\omega_{\mathfrak{z}}\sum_{m=0}^{\infty} \;\sideset{}{^*}\sum_{\mathfrak{b}\subseteq\mathcal{O}_{\Q(\mathfrak{z})}}\frac{C_{2k}\left(\mathfrak{b},m\right)}{N(\mathfrak{b})^{k}}\left(\frac{N(\mathfrak{b})}{\mathfrak{z}_2}\right)^j \left(2\pi i 
m\right)^{r}e^{\frac{2\pi m \mathfrak{z}_2}{N(\mathfrak{b})}} e^{2\pi i m z}.
$$
\end{theorem}
Using Theorem \ref{genFC2} instead of Theorem \ref{genFC}, we are now ready to prove Theorem \ref{Fouriercoefficients}.
\begin{proof}[Proof of Theorem \ref{Fouriercoefficients}]
The proof follows the proof of Theorem \ref{thm:FourierCoeff}, except that we plug Theorem \ref{genFC2} into \eqref{fasH} instead of Theorem \ref{genFC} to compute the Fourier coefficients of $H_{2k+2n,j}$, noting that $H_{2k+2n,j}(\tau_0,z)$ and $F_{2k+2n,j,n-j}(\tau_0,z)$ are both identically zero if $\omega_{\tau_0}\nmid k+n$.  
\end{proof}

\section{Proof of Theorem \ref{QuasiTheorem}}\label{sec:Quasi}
In this section, we explain identities between linear combinations of polar Maass forms with different eigenvalues and $\widehat{E}_2^{m}f$ for weight $2-2k$ meromorphic cusp forms $f$.  Furthermore, we show how to use these identities 
to determine the Fourier coefficients of $E_{2}^{m}f$. The following generalizes Theorem \ref{QuasiTheorem}.
\begin{theorem}\label{thm:EpowFCgen}
If $f\in\SS_{2-2k}$, $m\in\N_0$, and $n\in\N_0$ satisfies $2-2k+2n<0$, then 
the $m$th Fourier coefficient of $E_2^{n}f$ is a linear combination (for some $j,r\in \N_0$, $\tau_l\in \H$, and $k_l\in 2\N$ satisfying $k_l\geq 2k$) of 
\begin{equation}\label{eqn:EpowFCgen}
\sum_{\substack{c,d\in \Z\\ (c,d)=1}} \left(\frac{\left|c\tau_l+d\right|^2}{v_l }\right)^{j}\left(2\pi i m\right)^r B_{k_l,c,d}\left(\tau_j, m\right).
\end{equation}
\end{theorem}
\begin{remark}
The method used to prove Theorem \ref{thm:EpowFCgen} is effective in the sense that the Fourier coefficients can be computed in individual cases.  In order to demonstrate how to determine Fourier coefficients, we explicitly work out an example in Section \ref{5.1}. 
\end{remark}
Before proving Theorem \ref{thm:EpowFCgen}, we recall that the Fourier coefficients of arbitrary meromorphic cusp forms are computed in Theorem \ref{thm:FourierCoeff} and then use the following proposition to rewrite $E_2^nf$ in terms 
of $f$ multiplied by lower powers of $E_2$.
\begin{proposition}\label{prop:E2mero}
If $f\in \SS_{2-2k}$ and $n\in\N_0$ satisfies $2-2k+2n<0$, then 
\begin{equation}\label{eqn:gdef}
F_n(z):=\sum_{l=0}^{n} (-1)^{l}\binom{n}{l} \frac{(2k-2n-1)!}{(2k-2n-1+l)!} R_{2-2k+2n-2l}^{l}\left(\left(\frac{\pi}{3}\right)^{n-l}\widehat{E}_{2}^{n-l}(z)f(z)\right)\in \SS_{2-2k+2n}.
\end{equation}
\end{proposition}
\begin{proof}[Proof of Proposition \ref{prop:E2mero}]
Note that the function $F_n$ satisfies weight $2-2k+2n$ modularity by construction.  Furthermore, inspecting its Fourier expansion implies that for $F\in \mathbb{M}_{2-2k+2n}$ with $2-2k+2n<0$, $F\in \SS_{2-2k+2n}$ if and only if $F$ is bounded towards $i\infty$.  Since $f\in \SS_{2-2k}$, the right-hand side of \eqref{eqn:gdef} is bounded towards $i\infty$ termwise, so it remains to show that $F_n$ is meromorphic. For ease of notation, we let $\mathcal{G}_n:=(\frac{\pi}{3})^n \widehat{E}_2^nf$ and apply the lowering operator $L_{2-2k+2n}:=-2iy^2\frac{\partial}{\partial\overline{z}}$ to each term in \eqref{eqn:gdef}.  Recalling that 
\begin{equation}\label{eqn:E2lower}
L_{2}\left(\widehat{E}_2\right) =\frac{3}{\pi},
\end{equation}
we obtain, since $f$ is meromorphic, 
\begin{equation}\label{eqn:lowerE2hatf}
L_{2-2k+2n}
\left(\mathcal{G}_n\right) = n \mathcal{G}_{n-1}.
\end{equation}
This yields the image of the $l=0$ term in \eqref{eqn:gdef} and we are left to apply lowering to the other terms in \eqref{eqn:gdef}. Repeatedly commuting the lowering and raising operators via  (see (3.3) of \cite{BruinierFunke}) 
\begin{equation}\label{eqn:LR}
L_{\kappa}\circ R_{\kappa-2} + (\kappa-2) = -\Delta_{\kappa-2}= R_{\kappa-4}\circ L_{\kappa-2},
\end{equation}
we obtain that
$$
L_{2-2k+2n}\circ R_{2-2k+2n-2l}^{l}\!\left(\mathcal{G}_{n-l}\right)\! = l(2k-2n+l-1)R_{2-2k+2n-2l}^{l-1}\!\left(\mathcal{G}_{n-l}\right) + R_{-2k+2n-2l}^{l}\circ L_{2-2k+2n-2l}\!\left(\mathcal{G}_{n-l}\right).
$$
Using \eqref{eqn:lowerE2hatf}, 
this
 gives
\begin{multline}\label{eqn:lowerFn}
L_{2-2k+2n}\left(F_n\right) =  n \mathcal{G}_{n-1} + \sum_{l=1}^{n}(-1)^{l}\binom{n}{l} \frac{(2k-2n-1)!l}{(2k-2n-2+l)!} R_{2-2k+2n-2l}^{l-1}
\left(
\mathcal{G}_{n-l}
\right)\\
+ \sum_{l=1}^{n-1}(-1)^{l}\binom{n}{l} \frac{(2k-2n-1)!(n-l) }{(2k-2n-1+l)!}
R_{-2k+2n-2l}^{l}\left(
\mathcal{G}_{n-1-l}
\right).
\end{multline}
Shifting $l\mapsto l+1$ in the middle sum in \eqref{eqn:lowerFn} and using the fact that $\binom{n}{l}l = n\binom{n-1}{l-1}$, \eqref{eqn:lowerFn} implies that $L_{2-2k+2n}\left(F_n\right) = 0,$ and therefore $F_n$ is meromorphic.  
\end{proof}

Before proving Theorem \ref{thm:EpowFCgen}, we require an additional lemma about the meromorphic parts of the images of almost meromorphic modular forms under (repeated) raising.
\begin{lemma}\label{lem:raisemp}
Suppose that $f(z)=\sum_{j=0}^r y^{-j} f_j(z)$ is a weight $\kappa$ almost meromorphic modular form and $l\in\N_0$.  
Then the meromorphic part of $R_{\kappa}^{l}\left(f(z)\right)$ equals $(2i)^{l}\frac{\partial^{l}}{\partial z^{l}} f_0(z).$
\end{lemma}
\begin{proof}
Applying raising, we see that
$$
 R_{\kappa}\left(\sum_{j=0}^r y^{-j} f_j(z)\right) = \sum_{j=0}^{r+1} y^{-j} g_j(z),
$$
where $g_0:=2if_0'$ and
$g_j:=\left(\kappa-j+1\right)f_{j-1} + 2i f_j'$ for $j\in\N$.  Note that $g_j$ is meromorphic, and 
hence $g_0$ is the meromorphic part of $R_{\kappa}(\sum_{j=0}^r y^{-j} f_j(z))$.  By induction, the claim follows.
\end{proof}

We are now ready to prove Theorem \ref{thm:EpowFCgen}.
\begin{proof}[Proof of Theorem \ref{thm:EpowFCgen}]

We show the result by induction on $n$.  The statement for $n=0$ is Theorem \ref{thm:FourierCoeff}.  By Proposition \ref{prop:E2mero}, the functions $F_n$, defined in \eqref{eqn:gdef}, are elements of $\SS_{2-2k+2n}$.  By Theorem \ref{thm:FourierCoeff}  and Proposition \ref{prop:basis}, its Fourier expansion is  a linear combination of forms of the shape \eqref{eqn:EpowFCgen}.  Since $f$ and $\widehat{E}_2$ are both almost meromorphic modular forms and  raising preserves the shape \eqref{eqn:almostmeroexp}, there exist $r_{l}\in \N_0$ and meromorphic functions $f_{l,j}$ for $j=0,\dots,r_{l}$ such 
that 
$$
(-1)^{l}\binom{n}{l}\frac{(2k-2n-1)!}{(2k-2n-1+l)!} 
R_{2-2k+2n-2l}^{l}
\left(\left(\frac{\pi}{3}\right)^{n-l}\widehat{E}_{2}^{n-l}(z)f(z)\right)=\sum_{j=0}^{r_{l}} y^{-j} f_{l,j}(z).
$$
Taking $R:=\max_{l} r_{l}$ and defining $f_{l,j}:=0$ for $j>r_{l}$, \eqref{eqn:gdef} becomes
\begin{equation}\label{eqn:Falmostexp}
F_n(z)=\sum_{l=0}^{n} \sum_{j=0}^{R} y^{-j} f_{l,j}(z)=\sum_{j=0}^{R} y^{-j} \sum_{l=0}^{n} f_{l,j}(z).
\end{equation}
We next show that the terms in \eqref{eqn:Falmostexp} with $j>0$ vanish.  
Since the left-hand side is annihilated by the lowering operator,  lowering $J\in\N$ times yields that
$$
\sum_{j=J}^{R} \frac{j!}{(j-J)!} y^{J-j} \sum_{l=0}^n f_{l,j}(z)=0.
$$
Choosing $J=R$ yields a single non-zero term in the outer sum, so we conclude that $\sum_{l=0}^n f_{l,R}$ vanishes identically.  Plugging this back into the $J=R-1$ case, the outer sum in that case now becomes a single term, and we 
iteratively conclude that $ \sum_{l=0}^{n}f_{l,j}\equiv 0$ for all $j\in\N$.  
Hence \eqref{eqn:Falmostexp} simplifies to 
$$
f_{0,0} = F_n-\sum_{l=1}^{n} f_{l,0}.
$$
Since the meromorphic part of $\widehat{E}_2^{n}f$ equals $E_2^{n}f=(\frac{3}{\pi})^nf_{0,0}$ and by Theorem \ref{thm:FourierCoeff} the Fourier expansion of $F_n$ is a linear combination of forms of the shape \eqref{eqn:EpowFCgen}, it remains to show that all $f_{l,0}$ have expansions of the form 
 \eqref{eqn:EpowFCgen}.  By Lemma \ref{lem:raisemp}, one obtains that the meromorphic part of $R_{2-2k+2n-2l}^{l}((\frac{\pi}{3})^{n-l}\widehat{E}_{2}^{n-l}(z)f(z))$ equals $(2i)^{l}(\frac{\pi}{3})^{n-l}\frac{\partial^{l}}{\partial z^{l}}(E_2^{n-l}(z)f(z))$.  Using induction,  the forms $E_2^{n-l}f$ for $1\leq l\leq n$ all have Fourier coefficients which are linear combinations of \eqref{eqn:EpowFCgen}.  Since 
differentiation preserves
 the shape \eqref{eqn:EpowFCgen}, we conclude the claim.
\end{proof}

\section{Examples}\label{sec:examples}

\subsection{Preliminary calculations for meromorphic forms}\label{5.1}
In this section, we demonstrate how to explicitly compute the Fourier coefficients of a meromorphic cusp form given only the principal part of its Laurent expansion.  The idea is to use the principal part to rewrite the form as a linear combination of the basis elements $R_{2k,\mathfrak{z}}^n\left[H_{2k}(\mathfrak{z},z)\right]_{\mathfrak{z}=\tau_l}$ (with $\tau_l\in\H$) given in Proposition \ref{prop:basis} and then employ Proposition \ref{prop:Rtaugen} and Theorems \ref{genFC} and \ref{genFC2} to determine the Fourier coefficients of these basis elements.

The main step is to determine the representation of a meromorphic cusp form in terms of the above basis, given only the principal part of its Laurent series expansion.  To compare the principal parts, we use the following explicit version of Lemma \ref{lem:Rprincpart}.
\begin{lemma}\label{lem:Rprincpartexplicit}
The principal part of $R_{2k,\mathfrak{z}}^n\left[H_{2k}(\mathfrak{z},z)\right]_{\mathfrak{z}=\tau_0}$ at $z=\tau_0\in\mathbb{H}$ is given by 
$$
\widetilde{\varepsilon}_{2k+2n}\left(\tau_0\right)n!\sum_{j=0}^{n} \frac{(2k+n-1)!}{(2k-1+j)!(n-j)!} \frac{(2i)^{j}}{v_0^{n-j}}  \frac{1}{\left(z-\tau_0\right)^{j+1}},
$$
where $\widetilde{\varepsilon}_{2k}$ is defined in \eqref{eqn:Resztau}.
\end{lemma}

\begin{proof}
A straightforward calculation, following the discussion preceding (3a.10) of \cite{Pe1}, shows that for all $\kappa\in 2\N$, $j\in \N_0$, and $\tau_0\in \H$,
\begin{equation*}
\lim_{z\to \tau_0} \left(z-\tau_0\right)H_{\kappa,j}\left(\tau_0,z\right) =\widetilde{\varepsilon}_{\kappa}\left(\tau_0\right) \mathfrak{z}_2^{-j}.
\end{equation*}
Using Proposition \ref{prop:Rtaugen} and taking $n-j$ derivatives of the principal part of $H_{2k+2n,j}$ yields
$$
R_{2k,\mathfrak{z}}^{n}\left[H_{2k}(\mathfrak{z},z)\right]_{\mathfrak{z}=\tau_0}=  \widetilde{\varepsilon}_{2k+2n} \left(\tau_0\right) n! \sum_{j=0}^{n} \frac{(2k+n-1)!}{(2k+n-1-j)!j!} \frac{(2i)^{n-j}}{v_0^{j}} \frac{1}{\left(z-\tau_0\right)^{n-j+1}}  +O(1).
$$
The change of variables $j\to n-j$ then gives the statement of the lemma.
\end{proof}

We next assume that the only poles of $f\in \SS_{2-2k}$ occur at $\tau_0=i$ modulo $\SL_2(\Z)$ and that its principal part around $i$ is given by 
\begin{equation}\label{eqn:P0def}
P_i(z):=\frac{\alpha}{(z-i)^4} + \frac{\beta}{(z-i)^3} + \frac{\gamma}{(z-i)^2} + \frac{\delta}{z-i}
\end{equation}
with $\alpha\neq 0$.  Note that by \eqref{eqn:ellexp}, $\alpha\neq 0$ implies that $k\equiv 1\pmod{2}$, so that $\widetilde{\varepsilon}_{2k+2}(i)=\widetilde{\varepsilon}_{2k+6}(i)=i/\pi$ and $\widetilde{\varepsilon}_{2k}(i)=\widetilde{\varepsilon}_{2k+4}(i)=0$.  By Lemma \ref{lem:Rprincpartexplicit}, we have
\begin{equation*}
f(z)-\frac{\alpha \pi }{48}R_{2k,\mathfrak{z}}^3\left[H_{2k}(\mathfrak{z},z)\right]_{\mathfrak{z}=i}= \frac{\beta+\frac{2k+2}{2}i\alpha}{(z-i)^3} +\frac{\gamma + 
\frac{(2k+1)(k+1)}{4}
\alpha}{(z-i)^2} + \frac{\delta - 
\frac{k(2k+1)(k+1)}{12}
i\alpha}{z-i}+O(1).
\end{equation*}
Note that the congruence condition in the elliptic expansion \eqref{eqn:ellexp} implies that $\beta=-
(k+1)
i\alpha$ .  
By \eqref{eqn:ellexp}, we thus conclude that
\begin{equation}\label{eqn:fexpandLaur}
f(z)= \frac{\alpha \pi }{48} R_{2k,\mathfrak{z}}^3\left[H_{2k}(\mathfrak{z},z)\right]_{\mathfrak{z}=i} 
-\frac{ \pi}{2}
 \left(\gamma + 
\frac{(2k+1)(k+1)}{4}
\alpha\right)R_{2k,\mathfrak{z}}\left[H_{2k}(\mathfrak{z},z)\right]_{\mathfrak{z}=i}.
\end{equation}
Theorem \ref{Fouriercoefficients} then yields the Fourier coefficients of $f$.  In the next section, we explicitly compute the principal part of $1/E_6^4$ to obtain its Fourier expansion.  

\subsection{The example $1/E_6^4$}
We now work out the details for the specific example $1/E_6^4$.  
\begin{theorem}\label{thm:1/E6^4}
For $m\in\N_0$,
 the $m$th Fourier coefficient of $1/E_6^4$ is given by 
\begin{multline*}
\sideset{}{^*}\sum_{\mathfrak{b}\subseteq\mathcal{O}_{\Q(i)}}\Bigg(\frac{
1}{2\pi^3E_4^8(i)} \frac{C_{32}\left(\mathfrak{b},m\right)}{N(\mathfrak{b})^{16}}\sum_{j=0}^{3} \binom{28}{j}\frac{
\left(4\pi  m\right)^{3-j}}{(3-j)!}N(\mathfrak{b})^j 
\\
-\frac{28}{27\pi E_4^7(i)}\frac{C_{28}\left(\mathfrak{b},m\right)}{N(\mathfrak{b})^{14}} \sum_{j=0}^{1}\binom{26}{j}\left(
4\pi m\right)^{1-j}N(\mathfrak{b})^j 
\Bigg)e^{\frac{2\pi m}{N(\mathfrak{b})}}.
\end{multline*}
\end{theorem}
\begin{remark}
The special case of $m=0$ yields the interesting identity
$$
\sideset{}{^*}\sum_{\mathfrak{b}=(ci+d)\subseteq\mathcal{O}_{\Q(i)}}\frac{1}{N(\mathfrak{b})^{13}}\left(9\cos\left(32\arctan\left(\frac{c}{d}\right)\right)-4\pi^2 E_4(i) \cos\left(28\arctan\left(\frac{c}{d}\right)\right) \right)=
\frac{27\pi^3 E_4^8(i)}{182}.
$$
\end{remark}

In order to prove Theorem \ref{thm:1/E6^4}, we first express $1/E_6^4$ in terms of the basis from Proposition \ref{prop:basis}.
\begin{lemma}\label{lem:1/E6^4rewrite}
We have 
$$
\frac{1}{E_6^4(z)} =\frac{1}{48\pi^3E_4^8(i)}R_{26,\mathfrak{z}}^3\left[H_{26}(\mathfrak{z},z)\right]_{\mathfrak{z}=i}-\frac{7}{27 \pi E_4^7(i)} R_{26,\mathfrak{z}}\left[H_{26}(\mathfrak{z},z)\right]_{\mathfrak{z}=i}.
$$
\end{lemma}
\begin{proof}
We write the principal part of the Laurent series
 of $1/E_6^4$ around $z=i$ as in \eqref{eqn:P0def}.
By \eqref{eqn:fexpandLaur}, it remains to explicitly compute $\alpha$ and $\gamma$. For this, we first recall that by \eqref{eqn:ellexp}, 
$1/E_6^4$ has an elliptic expansion of the form
$$
\frac{1}{E_6^4(z)} =(z+i)^{24}\sum_{\substack{m\geq -4\\m\equiv 0\pmod{2} }}b_m\left(\frac{z-i}{z+i}\right)^m.
$$

We next write $\alpha, \beta, \gamma$, and $\delta$ in \eqref{eqn:P0def} in terms of the coefficients $
b_{-4}
$ and $
b_{-2}
$.
Expanding $(z+i)^{m}=(z-i+2i)^{m}$ with the Binomial Theorem, we obtain
\begin{align}
\label{eqn:alphabetaE64} \alpha &= 2^{28}b_{-4},& \beta & = -14i\alpha,\\
\label{eqn:deltaE64} \gamma&= -2^{26} b_{-2}-378\cdot 2^{26}b_{-4},&\delta&= -13i\gamma -819i\alpha.
\end{align}
Noting that $E_6(i)=0$, we have 
(see (30) of \cite{RamanujanArithmetic})
\begin{equation}\label{eqn:E6'i} 
E_6'(i)=\pi i \left(E_2(i)E_6(i) - E_4^2(i)\right)=-\pi i E_4^2(i)
\end{equation}
and hence 
\begin{equation}\label{eqn:alphaeval}
\alpha=\lim_{z\to i} \frac{(z-i)^4}{E_6^4(z)}= \frac{1}{E_6'(i)^4} = \frac{1}{\pi^4E_4^8(i)}.
\end{equation}

The calculation for $\gamma$ is slightly more involved.  We 
apply L'Hospital's rule $6$ times to obtain
\begin{equation}\label{eqn:gamma1}
\gamma=\lim_{z\to i}  (z-i)^2\left(\frac{1}{E_6^4(z)}- \frac{\alpha}{(z-i)^4} -\frac{\beta}{(z-i)^3}\right)
=\frac{-\alpha \frac{\partial^6}{\partial z^6}\left[ E_6^4(z)\right]_{z=i} - 6\beta\frac{\partial^5}{\partial z^5} \left[E_6^4(z)\right]_{z=i}}{720 E_6'(i)^4}.
\end{equation}
By \eqref{eqn:alphabetaE64}, we have
$$
\beta=-14i\alpha = -\frac{14i}{\pi^4E_4^8(i)},
$$
while
a straightforward 
but lengthy calculation yields that
\begin{align*}
\frac{\partial^5}{\partial z^5}\left[E_6^4(z)\right]_{z=i}&=240E_6'(i)^3E_6''(i),\\
\frac{\partial^6}{\partial z^6}\left[E_6^4(z)\right]_{z=i}&= 480 E_6'(i)^3E_6'''(i)+1080E_6'(i)^2E_6''(i)^2.
\end{align*}
Hence we conclude from (\ref{eqn:gamma1}) 
that
\begin{equation}\label{eqn:gamma}
\gamma =-\frac{\alpha}{6E_6'(i)^2}\left(4 E_6'(i) E_6'''(i) +9E_6''(i)^2 -168iE_6'(i)E_6''(i)\right).
\end{equation}

This reduces the problem to computing the values of iterated derivatives of $E_6$ at $i$.
Similar to \eqref{eqn:E6'i}, one uses the formulas in (30) of \cite{RamanujanArithmetic} for $E_2'$, $E_4'$, and $E_6'$ to obtain 
(the other derivatives are needed as intermediaries)
\begin{align}
\label{eqn:E2i}
E_2(i)&=\frac{3}{\pi},& E_2'(i)&=\frac{3 i}{2\pi} - \frac{\pi i}{6}E_4(i),& E_2''(i)&= -\frac{3}{2\pi} +\frac{\pi}{2}E_4(i),\\
\nonumber E_4'(i)&
= 2
iE_4(i),& E_4''(i)&=-5E_4(i)-\frac{5\pi^2}{9}E_4^2(i),\\
\label{eqn:E6''i}E_6''(i)&=7\pi E_4^2(i),
&E_6'''(i)&
=\frac{7\pi^3i}{9}
E_4^3(i)+42\pi i E_4^2(i).
\end{align}
Plugging \eqref{eqn:E6'i} 
and 
\eqref{eqn:E6''i} into \eqref{eqn:gamma} then yields
\begin{equation}\label{eqn:gammaeval}
\gamma =\frac{14\pi^2}{27}\alpha E_4(i) -\frac{189}{2}\alpha.
\end{equation}
The claim then follows by substituting \eqref{eqn:alphaeval} and \eqref{eqn:gammaeval} into \eqref{eqn:fexpandLaur}.
\end{proof}
\begin{proof}[Proof of Theorem \ref{thm:1/E6^4}]
The theorem follows directly from Lemma \ref{lem:1/E6^4rewrite} and Theorem \ref{Fouriercoefficients}.
\end{proof}

\subsection{The example $E_2/E_6^4$}

To demonstrate how to use Theorem \ref{QuasiTheorem}, we explicitly compute the coefficients of $E_2/E_6^4$.

\begin{theorem}\label{thm:E2/E6^4}
For $m\in\N_0$, the $m$th Fourier coefficient of $E_2/E_6^4$ is given by
\begin{multline*}
\sideset{}{^*}\sum_{\mathfrak{b}\subseteq\mathcal{O}_{\Q(i)}}\left(\vphantom{ \sum_{j=0}^{4} \frac{\binom{4}{j} }{(27-j)!} \left(\frac{N(\mathfrak{b})}{2}\right)^j  }\right.\frac{3C_{32}\left(\mathfrak{b},m\right)}{2\pi^4E_4^8(i)N(\mathfrak{b})^{16}} 
\sum_{j=1}^{4}
\binom{27}{j-1}
\frac{
\left(4\pi  m\right)^{4-j}}{(4-j)!} N(\mathfrak{b})
^j\\
-\frac{ C_{28}\left(\mathfrak{b},m\right)}{
468\pi^2E_4^7(i)N(\mathfrak{b})^{14}} \sum_{j=0}^{2}
\binom{26}{j}\left(59j-78\right)\frac{(4\pi m)^{2-j}}{(2-j)!}N(\mathfrak{b})^j
-\frac{47}{162 E_4^6(i)}\frac{C_{24}\left(\mathfrak{b},m\right)}{N(\mathfrak{b})^{12}} \left. \vphantom{ \sum_{j=0}^{4} \frac{\binom{4}{j}}{(27-j)!} \left(\frac{N(\mathfrak{b})}{-2i}\right)^j  }\right)e^{\frac{2\pi m}{N(\mathfrak{b})}}.
\end{multline*}
\end{theorem}

In order to prove 
Theorem \ref{thm:E2/E6^4},
 we first rewrite $E_2/E_6^4$ as a linear combination of derivatives of $1/E_6^4$ 
and meromorphic modular forms.
\begin{lemma}\label{lem:E2/E6^4}
We have 
$$
\frac{E_2(z)}{E_6^4(z)} = \frac{R_{24,\mathfrak{z}}^4\left[H_{24}(\mathfrak{z},z)\right]_{\mathfrak{z}=i}}{
384
\pi^4 E_4^8(i)}
-
  \frac{5R_{24,\mathfrak{z}}^2\left[H_{24}(\mathfrak{z},z)\right]_{\mathfrak{z}=i}}{
432
\pi^2 E_4^7(i) } - \frac{
47
H_{24}(i,z)}{648 E_4^6(i)}+\frac{i}{4\pi}\frac{\partial}{\partial z} \frac{1}{E_6^4(z)}.
$$
\end{lemma}

\begin{proof}
Using that $F_1$ is meromorphic by Proposition \ref{prop:E2mero}, Lemma \ref{lem:raisemp} implies that 
$$
F_1(z)=\frac{\pi}{3}E_2(z)f(z)-
\frac{i}{k-1}
 f^{\prime}(z).
$$
In the special case that $f=1/E_6^4$, we have $
k=13
$ and 
thus obtain
\begin{equation}\label{eqn:E2/E6^4F1}
\frac{E_2(z)}{E_6^4(z)} = 	\frac{3}{\pi} F_1(z) +\frac{i}{4\pi}\frac{\partial}{\partial z} \frac{1}{E_6^4(z)}.
\end{equation}
We now
 proceed as in the proof of \eqref{eqn:fexpandLaur} to rewrite $
F_1
$
 in the basis given in Proposition \ref{prop:basis}.  For this, we determine the principal part of $F_1$ at $z=i$.

As computed above, the principal part of $1/E_6^4$ is $P_i$ defined in \eqref{eqn:P0def} with $\alpha$ as in \eqref{eqn:alphaeval}, $\gamma$ given by \eqref{eqn:gammaeval}, $\beta=-14i\alpha$, and $\delta=-13i\gamma-819i\alpha$ (see \eqref{eqn:alphabetaE64} and \eqref{eqn:deltaE64}).  Taking the meromorphic part of both sides of \eqref{eqn:gdef}, the principal part of $F_1$ agrees with the principal part of 
\begin{multline*}
\frac{\pi}{3}E_2(z)P_i(z) -\frac{i}{12}\frac{\partial}{\partial z} P_i(z)=\frac{i\alpha}{3(z-i)^5} + \frac{\frac{\pi}{3}E_2(i)\alpha +\frac{i}{4}\beta}{(z-i)^4} +\frac{\frac{\pi}{3}E_2(i)\beta + \frac{\pi}{3}E_2'(i)\alpha+\frac{i}{6}\gamma}{(z-i)^3}\\
 + \frac{\frac{\pi}{3}E_2(i)\gamma + \frac{\pi}{3}E_2'(i)\beta+\frac{\pi}{3}\frac{E_2''(i)\alpha}{2}+\frac{i}{12}\delta}{(z-i)^2} + \frac{\frac{\pi}{3}E_2(i)\delta + \frac{\pi}{3}E_2'(i)\gamma + \frac{\pi}{3}\frac{E_2''(i) \beta}{2} + \frac{\pi}{3}\frac{E_2'''(i)\alpha }{6}}{z-i}+O(1).
\end{multline*}

Comparing the principal parts with those in Lemma \ref{lem:Rprincpartexplicit}, $F_1$ is a linear combination of the functions $R_{24,\mathfrak{z}}^{j}\left[H_{24}(\mathfrak{z},z)\right]_{\mathfrak{z}=i}$ for $j=0,2,4$.  To compute the explicit linear combination, we simplify the principal part by plugging in $E_2(i)$, $E_2'(i)$, 
and $E_2''(i)$ 
from \eqref{eqn:E2i}, and 
\begin{equation*}
E_2'''(i) =-\frac{9i}{4\pi}+\frac{3\pi i}{2} E_4(i)+\frac{\pi^3 i}{12} E_4^2(i).
\end{equation*}

A lengthy but straightforward calculation, comparing the principal parts with those in Lemma \ref{lem:Rprincpartexplicit}, as in  the proof of Lemma \ref{lem:1/E6^4rewrite}, yields 
\begin{equation}\label{eqn:F1final}
F_1(z)=\frac{\pi \alpha R_{24,\mathfrak{z}}^4\left[H_{24}(\mathfrak{z},z)\right]_{\mathfrak{z}=i}}{1152} -  \frac{5\pi^3E_4(i) \alpha R_{24,\mathfrak{z}}^2\left[H_{24}(\mathfrak{z},z)\right]_{\mathfrak{z}=i}}{1296} - \frac{47\pi^5E_4^2(i)\alpha H_{24}(i,z)}{1944},
\end{equation}
where we used 
 \eqref{eqn:Resztau} to evaluate $\widetilde{\varepsilon}_{4j}(i)=i/\pi$ for all $j\in \N$.  Using \eqref{eqn:alphaeval} and plugging \eqref{eqn:F1final} into \eqref{eqn:E2/E6^4F1} then gives the claim. 
\end{proof}
We are now ready to prove 
Theorem \ref{thm:E2/E6^4}.
\begin{proof}[Proof of Theorem \ref{thm:E2/E6^4}]
We individually compute the Fourier coefficients of each term in Lemma \ref{lem:E2/E6^4}.  Differentiating 
Theorem \ref{thm:1/E6^4} 
gives 
that the $m$th Fourier coefficient of $\frac{i}{4\pi}\frac{\partial}{\partial z} 1/E_6^4 (z)$ equals
\begin{multline*}
\sideset{}{^*}\sum_{\mathfrak{b}\subseteq\mathcal{O}_{\Q(i)}}\Bigg(-\frac{C_{32}\left(\mathfrak{b},m\right)}{
16 \pi^4 E_4^8(i)N(\mathfrak{b})^{16}}  \sum_{j=0}^{3} \binom{28}{j} \frac{\left(4\pi m\right)^{4-j}}{
(3-j)!}
N(\mathfrak{b})^j
\\
+\frac{7C_{28}\left(\mathfrak{b},m\right)}{54\pi^2 E_4^7(i)N(\mathfrak{b})^{14}} \sum_{j=0}^{1}
\binom{26}{j} \left(
4\pi m\right)^{2-j}
N(\mathfrak{b})^j\Bigg)e^{\frac{2\pi m}{N(\mathfrak{b})}} .
\end{multline*}

The coefficients of the other terms are then directly determined 
by
 Theorem \ref{Fouriercoefficients}.   A lengthy but straightforward calculation combining like terms and simplifying yields the statement of the proposition. 
\end{proof}

\subsection{Proof of Theorem \ref{thm:E2^n/E10}}
In this section, we prove a uniform formula for powers of $\widehat{E}_2$ times meromorphic forms with simple poles at elliptic fixed points.
Before stating the general result, we note that $E_{10}$ has simple poles exactly at $i$ and $\rho$ modulo $\SL_2(\Z)$, and hence one may uniquely choose  $a_1,a_2\in \C$ such that
$$
\frac{1}{E_{10}(z)} = a_1 H_{12}(i,z) + a_2H_{12}(\rho,z).
$$
A direct calculation, using 
\begin{align*}
E_4(\rho)
&
=0,
&E_{10}'(i)&=E_4(i)E_6'(i)=-\pi i E_4^3(i),& E_{10}'(\rho)&=E_4'(\rho)E_6(\rho)=-\frac{2\pi i}{3}  E_6^2(\rho)
\end{align*}
 to compare
 the residues at $\mathfrak{z}=i$ and $\mathfrak{z}=\rho$ on both sides, yields 
\begin{equation}\label{a12}
a_1 =\frac1{E_4^3(i)},\qquad a_2 =\frac1{E_6^2(\rho)}.
\end{equation}

To state the general result from which Theorem \ref{thm:E2^n/E10} then follows, we recursively define 
\begin{equation}\label{eqn:Hkncompdef}
\mathcal{H}_{2k,j}(\mathfrak{z},z):=-\sum_{l=0}^{j-1} \binom{j}{l} (-y)^{l-j}\mathcal{H}_{2k,l}(\mathfrak{z},z)
+\sum_{l=0}^{j}
c_{k,l,j}
\binom{j}{l} R_{2k-2l,\mathfrak{z}}^l \circ R_{2-2k+2l,z}^{j-l} \left(\mathcal{H}_{2k-2l}(\mathfrak{z},z)\right).
\end{equation}
where $c_{k,l,j}:=\frac{(2k-l-j-2)!}{(2k-l-1)!}(2k-2l-1)$.

Since the functions $ \mathcal{H}_{2k-2l}$ satisfy weight $2k-2l$ modularity as functions of $\mathfrak{z}$ and weight $2-2k+2l$ modularity as functions of $z$ by Proposition \ref{prop:modulartau}, one inductively sees that the functions $\sum_{l=0}^{j}\binom{j}{l}(-y)^{l-j}\mathcal{H}_{2k,l}(\mathfrak{z},z)$  are also modular in both variables.  

\begin{theorem}\label{thm:E2^n}
Suppose that $\tau_1,\dots, \tau_r$ 
(with $r=1$ or $r=2$) are elliptic fixed points and $a_{1},\dots, a_r\in \C$ such that $\sum_{m=1}^{r}a_{m} \mathcal{H}_{2k}\left(\tau_{m},z\right)$ is meromorphic.  Then we have, for all $
k>j+1
$, 
\begin{equation}\label{eqn:E2gen}
 \sum_{l=0}^{j} \binom{j}{l} (-y)^{l-j}\sum_{m=1}^r a_{m} \mathcal{H}_{2k,l}\left(\tau_{m},z\right) = \left(\frac{\pi}{3}\right)^j \widehat{E}_2^j(z) \sum_{m=1}^r a_{m}H_{2k}\left(\tau_{m},z\right).
\end{equation}
Furthermore, for every $0\leq j<k-1$, we have
\begin{equation}\label{eqn:HkjE2pow}
\sum_{m=1}^r a_{m} \mathcal{H}_{2k,j}\left(\tau_{m},z\right) = \sum_{m
=1}^r a_{m} H_{2k,j}\left(\tau_{m},z\right) = \left(\frac{\pi}{3}\right)^j E_2^j(z) \sum_{m=1}^r a_{m} H_{2k}\left(\tau_{m},z\right).
\end{equation}
\end{theorem}

Before proving Theorem \ref{thm:E2^n}, we first relate $H_{2k,j}$ to $\mathcal{H}_{2k,j}$ by showing that they satisfy the same recurrence relation with the polar harmonic Maass form $\mathcal{H}_{2k-2l}$ replaced by its meromorphic part $H_{2k-2l}$.
\begin{proposition}\label{prop:Hkn}
For $j<k$, we have 
\begin{equation}\label{eqn:Hkn}
\sum_{l=0}^{j}\binom{j}{l}(-y)^{l-j}H_{2k,l}(\mathfrak{z},z)=\sum_{l=0}^{j}
c_{k,l,j}
\binom{j}{l} R_{2k-2l,\mathfrak{z}}^{l} \circ R_{2-2k+2l,z}^{j-l} \left(H_{2k-2l}(\mathfrak{z},z)\right).
\end{equation}
\end{proposition}
\begin{remark}
The right-hand side of \eqref{eqn:Hkn} gives a decomposition of the left-hand side into eigenfunctions under $\Delta_{2-2k+2j,z}$ with different eigenvalues.  
\end{remark}
\begin{proof}
We prove the result by induction on $j$.  The base case $j=0$ is trivial.  

We inductively assume that \eqref{eqn:Hkn} holds for $j$ and show the claim for $j+1$.  We begin by rewriting the left-hand side of \eqref{eqn:Hkn} as 
\begin{equation}\label{eqn:1/v-1/ypow}
\sum_{M\in \Gamma_{\infty}\backslash \SL_2(\Z)} \left(\frac{\sum_{l=0}^{j}
\binom{j}{l} 
(-y)^{l-j}\mathfrak{z}_2^{-l}}{1-e^{2\pi i (z-\mathfrak{z})}}\right)\bigg|_{2k,\mathfrak{z}} M=\sum_{M\in \Gamma_{\infty}\backslash \SL_2(\Z)} \left(\frac{\left(\frac{1}{\mathfrak{z}_2}-\frac{1}{y}\right)^{
j}}{1-e^{2\pi i (z-\mathfrak{z})}}\right)\bigg|_{2k,\mathfrak{z}} M.
\end{equation}

Now note that both sides of \eqref{eqn:Hkn} are almost meromorphic modular forms as functions of $\mathfrak{z}$.  Writing them both in the form of \eqref{eqn:almostmeroexp}, they are equal if and only if the corresponding meromorphic functions $f_j$ agree for all $j$. For this, by analytic continuation, it suffices to prove the claim for $\mathfrak{z}_2>\max(y,1/y)$.  Under this assumption, 
we may expand the geometric series in each term on the right-hand side of \eqref{eqn:1/v-1/ypow} as well as in the definition of $H_{2k-2l}$ on the right-hand side of \eqref{eqn:Hkn}.  We now apply induction, rewrite the left-hand side of \eqref{eqn:Hkn} with the right-hand side of \eqref{eqn:1/v-1/ypow} and then expand the geometric series.  We then prove \eqref{eqn:Hkn} for $j+1$ by applying raising operators in $\mathfrak{z}$ and $z$ to the right-hand sides of \eqref{eqn:Hkn} and \eqref{eqn:1/v-1/ypow}.  Commuting the raising operator in $\mathfrak{z}$ with the slash operator, the image of the action of $R_{2k,\mathfrak{z}}$ on the right-hand side of \eqref{eqn:1/v-1/ypow} equals
\begin{equation}\label{eqn:Rtaugen}
\sum_{n=0}^{\infty}\sum_{M\in \Gamma_{\infty}\backslash \SL_2(\Z)} \left(\left(\frac{1}{\mathfrak{z}_2}-\frac{1}{y}\right)^{j-1}\left(-\frac{j}{\mathfrak{z}_2^2} + 4\pi n\left(\frac{1}{\mathfrak{z}_2}-\frac{1}{y}\right) +\frac{2k}{\mathfrak{z}_2}\left(\frac{1}{\mathfrak{z}_2}-\frac{1}{y}\right)\right)e^{-2\pi i n\mathfrak{z}}\right)\bigg|_{2k+2,\mathfrak{z}} M e^{2\pi i n z}.
\end{equation}
Furthermore, applying $R_{2-2k+2j,z}$ to \eqref{eqn:1/v-1/ypow} yields 
\begin{equation}\label{eqn:Rzgen}
\sum_{n=0}^{\infty}\sum_{M\in \Gamma_{\infty}\backslash \SL_2(\Z)}\!\! \bigg(\!\left(\frac{1}{\mathfrak{z}_2}-\frac{1}{y}\right)^{j-1}\!\bigg(\frac{j}{y^2} - 4\pi n\!\left(\frac{1}{\mathfrak{z}_2}-\frac{1}{y}\right) +\frac{2-2k+2j}{y}\left(\frac{1}{\mathfrak{z}_2}-\frac{1}{y}\right)\!\bigg)e^{-2\pi i n\mathfrak{z}}\bigg)\bigg|_{2k,\mathfrak{z}} M e^{2\pi i n z}.
\end{equation}
After simplifying, the sum of \eqref{eqn:Rzgen} and \eqref{eqn:Rtaugen} (with $2k\to 2k-2$) equals
\begin{equation}\label{eqn:RHk}
(2k-2-j)\sum_{n=0}^{\infty}\sum_{M\in \Gamma_{\infty}\backslash \SL_2(\Z)} \left(\left(\frac{1}{\mathfrak{z}_2}-\frac{1}{y}\right)^{j
+1}e^{-2\pi i n\mathfrak{z}}\right)\bigg|_{2k,\mathfrak{z}} M e^{2\pi i n z},
\end{equation}
which is up to the constant $2k-2-j$ the right-hand side of \eqref{eqn:1/v-1/ypow} (and hence the left-hand side of \eqref{eqn:Hkn}) with $j\mapsto j+1$.  

We next apply raising in $\mathfrak{z}$ and $z$ to the right-hand side of \eqref{eqn:Hkn}.  Shifting $2k\mapsto 2k-2$ and $l\mapsto l-1$ and then applying $R_{2k-2,\mathfrak{z}}$ to it yields
\begin{equation}\label{eqn:RtauRgen}
\sum_{l=1}^{j+1}
c_{k-1,l-1,j}
\binom{j}{l-1} R_{2k-2l,\mathfrak{z}}^{l} \circ R_{2-2k+2l,z}^{j+1-l} \left(H_{2k-2l}(\mathfrak{z},z)\right).
\end{equation}
Applying $R_{2-2k+2j,z}$ to the right-hand side of \eqref{eqn:Hkn}, we obtain
\begin{equation}\label{eqn:RzRgen}
\sum_{l=0}^{j}c_{k,l,j}\binom{j}{l} R_{2k-2l,\mathfrak{z}}^{l} \circ R_{2-2k+2l,z}^{j+1-l} \left(H_{2k-2l}(\mathfrak{z},z)\right).
\end{equation}
Since $c_{k,l,j+1}=c_{k,l,j}/(2k-l-j-2)=c_{k-1,l-1,j}/(2k-l-1)$, the sum of \eqref{eqn:RtauRgen} and \eqref{eqn:RzRgen} (note that the sums in $l$ can be extended because the binomial coefficients vanish for the extra terms) then equals 
\begin{multline}\label{eqn:Rtosimplify}
\sum_{l=0}^{j+1}c_{k,l,j+1}\left( (2k-l-1)\binom{j}{l-1}+ (2k-l-j-2) \binom{j}{l}\right)\\
\times R_{2k-2l,\mathfrak{z}}^{l} \circ R_{2-2k+2l,z}^{j+1-l} \left(H_{2k-2l}(\mathfrak{z},z)\right).
\end{multline}
We then rewrite \eqref{eqn:Rtosimplify} using \eqref{eqn:binom} (with $2k\mapsto 2k-2n-2$ and then $j\mapsto l$ and $n\mapsto j$).  This yields $2k-2-j$ times the right-hand side of \eqref{eqn:Hkn} with $j\mapsto j+1$.  Comparing with \eqref{eqn:RHk} gives that \eqref{eqn:Hkn} holds for $\mathfrak{z}_2>\max(y,1/y)$.
\end{proof}

\begin{proof}[Proof of Theorem \ref{thm:E2^n}]
By Proposition \ref{prop:Hkn} and definition \eqref{eqn:Hkncompdef}, the left-hand side of \eqref{eqn:E2gen} is modular of weight $2-2k+2j$, while this is also true for the right-hand side by construction.  In order to prove the identity, we show that the difference of both sides is an element of $\H_{2-2k+2j}^{\operatorname{cusp}}$ as a function of $z$ and then prove that it has no principal part, implying that it is identically zero.

To see that that the difference is harmonic, we apply $\Delta_{2-2k+2j,z}$ to both sides.  Rewriting $\Delta_{2-2k+2j,z}$  with \eqref{eqn:LR} and recalling that  $\sum_{m=1}^r a_{m}\mathcal{H}_{2k}\left(\tau_{m},z\right)$ is meromorphic by assumption, \eqref{eqn:E2lower} implies that
\begin{equation}\label{eqn:E2genDelRHS}
\Delta_{2-2k+2j,z}\left(\frac{\pi}{3} \widehat{E}_2^j(z) \sum_{m=1}^r a_{m}H_{2k}\left(\tau_{m},z\right)\right)= -jR_{-2k+2j,z}\left(\widehat{E}_2^{j-1}(z)\sum_{m=1}^r a_{m}H_{2k}\left(\tau_{m},z\right)\right).
\end{equation}

We then substitute \eqref{eqn:Hkncompdef} into the left-hand side of \eqref{eqn:E2gen} in order to compute the image of $\Delta_{2-2k+2j,z}$ 
on \eqref{eqn:E2gen}
 which hence equals
\begin{equation}\label{eqn:E2genDelLHS1}
\sum_{l=0}^{j}
c_{k,l,j}
\binom{j}{l} \sum_{m=1}^ra_{m} R_{2k-2l,\mathfrak{z}}^l \left[ \Delta_{2-2k+2j,z}\left(R_{2-2k+2l,z}^{j-l} \left(\mathcal{H}_{2k-2l}\left(\mathfrak{z},z\right)\right)\right)\right]_{\mathfrak{z}=\tau_{m}}.
\end{equation}

We next simplify \eqref{eqn:E2genDelLHS1} by showing that $R_{2-2k+2l,z}^{j-l} \left(\mathcal{H}_{2k-2l}\left(\mathfrak{z},z\right)\right)$ is an eigenfunction under $\Delta_{2-2k+2j,z}$. For this, note that Proposition \ref{prop:modulartau} yields that for every $\kappa$ the function $z\mapsto \mathcal{H}_{\kappa}(\mathfrak{z},z)\in\H_{2-\kappa}^{\operatorname{cusp}}$, so in particular it is annihilated by $\Delta_{2-\kappa,z}$.  However, for any weak Maass form $\mathcal{F}$ of weight $\kappa$ with eigenvalue $\lambda$ we have 
\begin{equation}\label{eqn:DelEV}
\Delta_{\kappa+2}\left(R_{\kappa}(\mathcal{F})\right) = \left(\lambda+\kappa\right) R_{\kappa}(\mathcal{F}).
\end{equation}
Evaluating $\sum_{n=1}^{j-l} (2-2k+2l+2(n-1))=(j-l)(j-2k+l+1)$  hence gives
$$
\Delta_{2-2k+2j,z}\left( R_{2-2k+2l,z}^{j-l} \left( \mathcal{H}_{2k-2l}(\mathfrak{z},z)\right)\right) = (j-l)(j-2k+l+1) R_{2-2k+2l,z}^{j-l} \left( \mathcal{H}_{2k-2l}(\mathfrak{z},z)\right).
$$
Using $\binom{j}{l}(j-l) = j\binom{j-1}{l}$ and $c_{k,l,j-1}=(2k-l-j-1)c_{k,l,j}$ then yields that \eqref{eqn:E2genDelLHS1} equals 
$$
-jR_{-2k+2j,z}\Bigg(\sum_{l=0}^{j-1}c_{k,l,j-1}\binom{j-1}{l}R_{2k-2l,\mathfrak{z}}^l \!\!\left[R_{2-2k+2l,z}^{j-1-l} \!\left(\sum_{m=1}^r a_{m}\mathcal{H}_{2k-2l}(\mathfrak{z},z)\right)\right]_{\mathfrak{z}=\tau_{m}}\!\Bigg).
$$
We use \eqref{eqn:Hkncompdef} and then induction on $j$ in \eqref{eqn:E2gen} to rewrite this as 
\begin{multline}\label{eqn:E2genDelLHS}
-j R_{-2k+2j,z}\left( \sum_{l=0}^{j-1} \binom{j-1}{l} (-y)^{l+1-j}\sum_{m=1}^r a_{m} \mathcal{H}_{2k,l}\left(\tau_{m},z\right)\right)\\
=-j\left(\frac{\pi}{3}\right)^{j-1} R_{-2k+2j,z}\left(\widehat{E}_2^{j-1}(z)\sum_{m=1}^r a_{m} H_{2k}\left(\tau_{m},z\right)\right).
\end{multline}

Combining \eqref{eqn:E2genDelLHS} and \eqref{eqn:E2genDelRHS}, we see that the difference between the left-hand and right-hand sides of \eqref{eqn:E2gen} is a weight $2-2k+2j<0$ polar harmonic Maass form, while \eqref{eqn:xiH2} implies that it is an element of $\H_{2-2k+2j}^{\operatorname{cusp}}$.  
In order to prove the identity, it remains to show that the meromorphic part of this difference has no singularities in $\H$.  Rewriting
 the meromorphic part of the left-hand side of \eqref{eqn:E2gen} as the right-hand side of \eqref{eqn:1/v-1/ypow}, we see that it is bounded at each $z=\tau_{m}$.  The right-hand side of \eqref{eqn:E2gen} is also bounded at each $z=\tau_m$, since $\widehat{E}_2^j$ has a zero of order $j$ at each $\tau_{m}$, while $H_{2k}(\tau_{m},z)$ has only a simple pole at $z=\tau_{m}$.  Hence the difference is a negative-weight harmonic Maass form with no singularities, and thus vanishes identically, proving \eqref{eqn:E2gen}.

To obtain \eqref{eqn:HkjE2pow}, first note that the $j=0$ case is trivial.  We then apply induction to assume 
the first identity in \eqref{eqn:HkjE2pow} for $l<j$.
Rearranging \eqref{eqn:E2gen} and applying induction to solve for $\sum_{m=1}^r a_{m}\mathcal{H}_{2k,j}\left(\tau_{m},z\right)$, we see that 
\begin{equation}\label{eqn:tolower}
\sum_{m=1}^r a_{m}\mathcal{H}_{2k,j}\left(\tau_{m},z\right)= -\sum_{l=0}^{j-1} \binom{j}{l} (-y)^{l-j}\sum_{m=1}^r a_{m} H_{2k,l}\left(\tau_{m},z\right) +\left(\frac{\pi}{3}\right)^j \widehat{E}_2^j(z) \sum_{m=1}^r a_{m}H_{2k}\left(\tau_{m},z\right)
\end{equation}
is an almost meromorphic modular form.   By \eqref{eqn:E2genDelRHS} and the fact that $
z\mapsto
H_{2k,l}(\mathfrak{z},z)$ is meromorphic, lowering applied to the right-hand side of \eqref{eqn:tolower} yields
$$
j\sum_{l=0}^{j-1} \binom{j-1}{l}(-y)^{l-j+1}\sum_{m=1}^r a_{m} H_{2k,l}\left(\tau_{m},z\right)
 -j\left(\frac{\pi}{3}\right)^{j-1} \widehat{E}_2^{j-1}(z)\sum_{m=1}^r a_{m}H_{2k}\left(\tau_{m},z\right)=0,
$$
where we used \eqref{eqn:E2gen} (with the first identity in \eqref{eqn:HkjE2pow} inductively plugged in) in the last step.  By \eqref{eqn:tolower}, it follows that $\sum_{m=1}^r a_{m}\mathcal{H}_{2k,j}\left(\tau_{m},z\right)$ is meromorphic.  Denoting
$$
g_{l,j}(z):=\sum_{m=1}^r a_m 
c_{k,l,j}
\binom{j}{l} R_{2k-2l,\mathfrak{z}}^l \left[R_{2-2k+2l,z}^{j-l} \left(\mathcal{H}_{2k-2l}^-(\mathfrak{z},z)\right)\right]_{\mathfrak{z}=\tau_m}
$$
and noting that $
z\mapsto H_{2k,j}(\mathfrak{z},z)
$ is meromorphic, Proposition \ref{prop:Hkn} hence implies that 
\begin{equation}\label{eqn:Hkjdiff}
z\mapsto 
\sum_{m=1}^r a_{m}\mathcal{H}_{2k,j}(\tau_m,z)-\sum_{m=1}^r a_{m}H_{2k,j}(\tau_m,z)=\sum_{l=0}^{j}g_{l,j}(z)
\end{equation}
is meromorphic.  We next show that the right-hand side of \eqref{eqn:Hkjdiff} is zero, implying the first identity in \eqref{eqn:HkjE2pow}.  To do so, we first recognize each $g_{l,j}$ as an eigenfunction under $\Delta_{2-2k+2j,z}$, so that \eqref{eqn:Hkjdiff} may be seen as a decomposition into eigenspaces.  The projection of \eqref{eqn:Hkjdiff} into each eigenspace is then simultaneously an eigenfunction under $\Delta_{2-2k+2j,z}$ and meromorphic, from which we conclude 
that 
it is zero unless the eigenvalue is zero.  The projection to the eigenspace with eigenvalue zero we then show is both meromorphic and also the non-meromorphic part of a polar harmonic Maass form, implying that it vanishes.

To realize \eqref{eqn:Hkjdiff} as a decomposition into eigenspaces, we recall that since $\mathcal{H}_{2k-2l}^-$ is annihilated by $\Delta_{2-2k+2l,z}$, \eqref{eqn:DelEV} implies that $g_{l,j}$ is an eigenfunction under $\Delta_{2-2k+2j,z}$ with eigenvalue
$$
\alpha_{l,j}:=\sum_{n=0}^{j-l} \left(2-2k+2l+2(n-1)\right)=(j-l)(j-2k+l+1).
$$
Since $\sum_{l=0}^jg_{l,j}$ is meromorphic, it is in particular annihilated by $\Delta_{2-2k+2j,z}$.  Thus, for every $0\leq l_0\leq j$, the function 
$$
\prod_{\substack{0\leq l\leq j\\ \alpha_{l,j}\neq \alpha_{l_0,j}}}\left(\Delta_{2-2k+2j,z}-\alpha_{l,j}\right)\left(\sum_{l=0}^jg_{l,j}(z)\right) = \prod_{\substack{0\leq l\leq j\\ \alpha_{l,j}\neq \alpha_{l_0,j}}}\left(\alpha_{l_0,j}-\alpha_{l,j}\right)\sum_{\substack{0\leq l\leq j\\ \alpha_{l,j}=\alpha_{l_0,j}}}g_{l,j}(z)
$$
is both meromorphic (and hence in particular annihilated by $\Delta_{2-2k+2j,z}$) and an eigenfunction with eigenvalue $\alpha_{l_0,j}$ under $\Delta_{2-2k+2j,z}$.  If $\alpha_{l_0,j}\neq 0$,  then we conclude that 
$$
\sum_{\substack{0\leq l\leq j\\ \alpha_{l,j}=\alpha_{l_0,j}}} g_{l,j}(z)=0.
$$
Hence 
$$
\sum_{l=0}^j g_{l,j}=\sum_{\substack{0\leq l\leq j\\ \alpha_{l,j}=0}} g_{l,j},
$$
but $\alpha_{l,j}=0$ if and only if $l=j$ or $l=2k-j-1$.  Since $l\leq j<k-1$, we have $l<2k-j-2$, and thus the second case cannot occur.  Therefore
$$
\sum_{l=0}^j g_{l,j}(z)=g_{j,j}(z)=\sum_{m=1}^ra_m
c_{k,j,j}
R_{2k-2j,\mathfrak{z}}^j \left[\mathcal{H}_{2k-2j}^-(\mathfrak{z},z)\right]_{\mathfrak{z}=\tau_m}
$$
is meromorphic.  But by Proposition \ref{prop:modulartau}, $g_{j,j}$ is the non-meromorphic part of the polar harmonic Maass form
$$
\sum_{m=1}^ra_m
c_{k,j,j} R_{2k-2j,\mathfrak{z}}^j \left[\mathcal{H}_{2k-2j}(\mathfrak{z},z)\right]_{\mathfrak{z}=\tau_m}.
$$
Since the non-meromorphic part of a polar harmonic Maass form can only be meromorphic if it is zero, we conclude that the right-hand side of \eqref{eqn:Hkjdiff} is zero, implying the first identity in \eqref{eqn:HkjE2pow}.   

Substituting the first identity in \eqref{eqn:HkjE2pow} into \eqref{eqn:E2gen} and comparing the mermorphic parts on both sides yields the second identity in \eqref{eqn:HkjE2pow}.  

\end{proof}
We finally prove Theorem \ref{thm:E2^n/E10}.
\begin{proof}[Proof of Theorem \ref{thm:E2^n/E10}]

Since $1/E_{10}(z)=a_1H_{12}(i,z)+a_2H_{12}(\rho,z)$ is meromorphic (write $a_1$ and $a_2$ as in \eqref{a12}), 
\eqref{eqn:HkjE2pow} yields
$$
\left(\frac{\pi}{3}\right)^n\frac{E_2^n
(z)
}{E_{10}
(z)
}=\left(\frac{\pi}{3}\right)^nE_2^n (z) \left(a_1H_{12}(i,z)+a_2H_{12}(\rho,z)\right) =a_1H_{
12
,n}(i,z)+a_2H_{12,n}(\rho,z).
$$
Plugging in Theorem \ref{genFC2}, we then obtain the desired result.
\end{proof}

\end{document}